\documentclass[a4paper,twoside,reqno,final]{amsart}

\usepackage{url}
\usepackage{dune}
\usepackage{pgfplots}
\pgfplotsset{compat=1.13}

\newcommand{\new}[1]{{{#1}}}

\usepackage[english]{babel}
\usepackage{amsthm,amssymb,amsmath}
\usepackage{enumerate}
\usepackage{setspace}
\usepackage[verbose,section]{placeins}

\usepackage{tikz}
\usetikzlibrary{patterns}
\usetikzlibrary{decorations.pathreplacing}
\usetikzlibrary{arrows,positioning} 
\usepackage{tikz-qtree}

\usepackage{comment}
\usepackage{color}
\usepackage{xspace}
\usepackage{graphicx,epstopdf}
\usepackage{array}
\usepackage{multirow}
\usepackage{todonotes}
\usepackage{subcaption}

\usepackage{algorithm}
\usepackage[algo2e,ruled,algosection,linesnumbered]{algorithm2e}

\SetCommentSty{mycommfont}

\RequirePackage[notoday]{svninfo}
\usepackage[scrtime]{prelim2e}

\usepackage[bookmarks=true]{hyperref}

\numberwithin{figure}{section}
\numberwithin{equation}{section}

\renewenvironment{itemize}
{\begin{list}{$\bullet$}{\labelwidth0mm \leftmargin3mm %
  \itemsep0pt plus 0pt \topsep3pt \parsep1pt plus 4pt \labelsep2mm}}
{\end{list}}

\newtheorem{theorem}{Theorem}[section]

\newtheorem{corollary}[theorem]{Corollary}
\newtheorem{lemma}[theorem]{Lemma}

\theoremstyle{definition}
\newtheorem{definition}[theorem]{Definition}

\theoremstyle{remark}
\newtheorem{remark}[theorem]{Remark}
\newtheorem{example}[theorem]{Example}

\numberwithin{figure}{section}
\numberwithin{equation}{section}

\makeatletter
\newcommand*\rel@kern[1]{\kern#1\dimexpr\macc@kerna}
\newcommand*\widebar[1]{%
  \begingroup
  \def\mathaccent##1##2{%
    \rel@kern{0.8}%
    \overline{\rel@kern{-0.8}\macc@nucleus\rel@kern{0.2}}%
    \rel@kern{-0.2}%
  }%
  \macc@depth\@ne
  \let\math@bgroup\@empty \let\math@egroup\macc@set@skewchar
  \mathsurround\z@ \frozen@everymath{\mathgroup\macc@group\relax}%
  \macc@set@skewchar\relax
  \let\mathaccentV\macc@nested@a
  \macc@nested@a\relax111{#1}%
  \endgroup
}
\makeatother

\DeclareMathOperator{\conv}{conv}

\newcommand{\edge}{\ensuremath{E}}
\newcommand{\edges}{\ensuremath{\mathcal{E}}}

\newcommand{\elm}{\ensuremath{T}\xspace}

\let\Forall=\forall
\renewcommand{\forall}{\Forall\,}

\newcommand{\face}{\mathit{F}}
\newcommand{\faces}{\ensuremath{\mathcal{F}}}
\newcommand{\freefaces}{\mathbb{F}}

\newcommand{\grid}{\mathcal{T}}

\newcommand{\marked}{\mathcal{M}}

\newcommand{\N}{\ensuremath{\mathbb{N}}}

\newcommand{\normM}[2][]{%
  #1\vert\kern-0.9pt#1\vert\kern-0.9pt#1\vert #2
  #1\vert\kern-0.9pt#1\vert\kern-0.9pt#1\vert}
\newcommand{\NVB}{\textsf{NVB}\xspace}

\let\Paragraph=\paragraph
\renewcommand{\paragraph}[1]{\Paragraph{\textbf{#1.}}}

\newcommand{\R}{\ensuremath{\mathbb{R}}}
\newcommand{\RT}{\mathcal{RT}}

\newcommand{\vertices}{\ensuremath{\mathcal{V}}}

\begin{document}

\title[A Weak Compatibility Condition For Newest Vertex Bisection]
{\boldmath A Weak Compatibility Condition For Newest Vertex Bisection in any dimension}

\author[M. Alk\"amper]{Martin Alk\"amper$^a$}
\author[F. Gaspoz]{Fernando Gaspoz$^a$}
\author[R. Kl\"ofkorn]{Robert Kl\"ofkorn$^b$}

\urladdr{\url{www.ians.uni-stuttgart.de/nmh/}}
\email{alkaemper@ians.uni-stuttgart.de}
\email{fgaspoz@ians.uni-stuttgart.de}
\email{Robert.Kloefkorn@iris.no}

\address{$^A$ Institut f\"ur Angewandte Analysis und
  Numerische Simulation, Fachbereich Mathematik,
  Universit\"at Stuttgart,
  Pfaffenwaldring 57, D-70569 Stutt\-gart, Germany }%

\address{$^B$ International Research Institute of Stavanger,
  Thormoehlensgt. 55, 5006 Bergen, Norway}%


\date{\today}


\begin{abstract}
We define a weak compatibility condition for the Newest Vertex Bisection algorithm on simplex grids of any dimension and show that using this condition the iterative algorithm terminates successfully. Additionally we provide an $O(n)$ algorithm that renumbers any simplex grid to fulfil this condition. Furthermore we conduct experiments to estimate the distance to the standard compatibility and also the geometric quality of the produced meshes.
\end{abstract}

\maketitle

\section{Introduction\label{sec:intro}}

Dynamically adaptive conforming unstructured meshes based on simplices usually use either Newest Vertex Bisection (\NVB) \cite{CaKrNoSi:08,GaScSt:14,GaHeSi:15,Mitchell:16,AlkamperKlofkorn:17} or Longest Edge Bisection (LEB) \cite{Rivara:97, HaKoKr:10,ApCaHe:15,KoKrKr:08,Horst:97}. While \NVB is a mere topological construct, \new{where refinement} simply depends on an ordering of the vertices, LEB uses geometric information to always refine the longest edge. \new{While it is clear for \NVB, that it only produces a finite number of shape classes, for LEB this is non-trivial and requires some combinatorial effort \cite{ApCaHe:15}.} 
This is a strong argument for using \NVB and an overview on \NVB can be found in \cite{Mitchell:16}.

However, one big drawback of \NVB is its non-applicability to generic unstructured meshes in 3 (or more) dimensions, as it needs a compatibility condition between neighboring elements on the initial grid. Standard simplicial mesh generators, such as \gmsh \cite{gmsh:09} or \tetgen \cite{tetgen:15}, do not guarantee this condition. This has been an open problem for a long time 
and with this work, to the best of our knowledge, for the first time a solution to this problem is presented that is applicable in any dimension and relies only  on standard NVB elements (in contrast to \cite{ArMuPo:00} ). Also it does not multiply the number of elements by $\frac{1}{2}(d+1)!$, while halving each angle $d$ times, as in \cite{Kossaczky:94,Stevenson:08}, and it is surprisingly simple.

The standard (strong) compatibility condition has been developed by Traxler and Maubach \cite{Maubach:95,Traxler:97} and generalized by Stevenson \cite{Stevenson:08}. It ensures that every level of uniform refinement is conforming and that there is a bound on the effort of the conforming closure. There even is a bound on the effort for parallel computations \cite{AlkamperKlofkorn:17}.

We introduce a weaker compatibility condition that is applicable to generic unstructured simplicial grids. We gain \new{applicability} at the cost of losing some of the convenient properties obtained by the stronger compatibility condition. It generalizes the concept of a mesh being conformingly marked introduced by Arnold et al. in \cite{ArMuPo:00} and directly relates to what Stevenson calls compatibly divisible \cite{Stevenson:08}. We will show that \NVB terminates successfully using this condition.

Furthermore, we present an algorithm that is capable of \new{relabeling} any grid to be weakly compatible and has in principle an 
effort of $O(n)$, where $n$ is the number of elements in the grid. In our
implementation, however, the algorithm shows a complexity of $O(n \log n)$ which is related to a neighbor search 
that is carried out prior to the relabeling algorithm. In addition, the
algorithm is capable to recover a strongly compatible situation for some meshes.

\new{To estimate the effort of the conforming closure, we will investigate metrics that relate to distances to a strongly compatible situation, as a strongly compatible situation minimizes the effort by design. The first metric measures refinement propagation on the macro level and the second metric exploits} the locality of the definition of strong compatibility and counts the amount of faces where the corresponding neighbors do not fulfil the strong compatibility condition.

\section{Weak Compatibility Condition}\label{sec:weakcompcond}

\subsection{Introduction}
We shortly introduce \NVB following the notation of Stevenson \cite{Stevenson:08}, while the algorithm itself was originally introduced by Maubach and Traxler \cite{Maubach:95,Traxler:97}.

\begin{definition}[Newest Vertex Bisection]
We identify a $d$-dimensional simplex $\elm$ (the convex hull $\conv\{z_i : 0 \leq i \leq d\}$) with an ordering of the vertices $z_i$ and its type $0\leq t_\elm\leq d-1$, i.e., 
\[
\elm = [z_0,z_1,\ldots,z_d]_{t_\elm}.
\]
Then the Newest Vertex Bisection (\NVB) of the simplex $T$ is defined by its two children 
\[
\elm_0 = \left[z_0,\frac{z_0 + z_d}{2},z_1,\ldots,z_{d-1}\right]_{t_\elm + 1 \mod d}
\]
and
\[
 \elm_1 = \left[z_d,\frac{z_0 + z_d}{2},z_1,\ldots,z_{t_\elm},z_{d-1},\ldots,z_{t_\elm +1}\right]_{t_\elm + 1 \mod d}.
\]
We call the edge $\edge_\elm = \overline{z_0z_d}$ the \emph{refinement edge} of the simplex $\elm$.
\end{definition}

\begin{example}[\NVB in 3d]
\label{ex:facesoftetra}
In 3 dimensions the simplex $\elm = [z_0,z_1,z_2,z_3]_{t_\elm}
$
is of type $t_\elm \in \{0,1,2\}$.

Bisecting the refinement edge $\edge_\elm = \overline{z_0z_3}$ at its center $z_{03}=\frac{z_0+z_3}{2}$ leads to the two simplices
\begin{equation}
T_0 = [z_0,z_{03},z_1,z_2]_{t_\elm +1\mod 3}  \text{ and }   \begin{cases} T_1=[z_3,z_{03},z_2,z_1]_1 &  \text{if } t_\elm=0, \\
T_1=[z_3,z_{03},z_1,z_2]_2 &  \text{if } t_\elm=1, \\
T_1=[z_3,z_{03},z_1,z_2]_0 &  \text{if } t_\elm=2. \end{cases}
\end{equation}
The refinement edges of the children are $\edge_{\elm_0} = \overline{z_0z_2}$ and $\edge_{\elm_1} = \overline{z_3z_1}$, if $t_\elm=0$ and $\edge_{\elm_1} = \overline{z_3z_2}$, if $t_\elm \in \{1,2\}$.

From this we can directly deduce the initial refinement edge of all $4$ faces of the simplex, namely
\begin{align*}
\face_0 = \{z_0,z_1,z_2\}: \qquad &  \edge_{\face_0} = \overline{z_0z_2},\\
\face_1 = \{z_0,z_1,z_3\}: \qquad & \edge_{\face_1} = \overline{z_0z_3},\\
\face_2 = \{z_0,z_2,z_3\}:  \qquad& \edge_{\face_2} = \overline{z_0z_3},\\
\face_3 = \{z_1,z_2,z_3\}:  \qquad& \edge_{\face_3} = \begin{cases} \overline{z_1z_3} & \text{if } t_\elm = 0, \\ \overline{z_2z_3} & \text{else}. \end{cases}
\end{align*}
\end{example}

Notice that two 3d neighboring simplices will match if the initial refinement edge on their shared face coincides, otherwise the refinement algorithm will fail. 
Here follows the idea of Arnold et al. \cite{ArMuPo:00} to set for each face the initial refinement edge as the longest edge. From there they create the simplices containing said faces, \new{but then have to introduce non-standard initial elements and their respective refinement in order to  combat the problem that arises from the fact that the simplices are not necessarily \NVB simplices.}

\subsection{Some \NVB Properties}

\new{By creating a finite number of similarity classes  \NVB creates a sequence of non-degenerating simplices  (see~\cite{Traxler:97,Maubach:95}).
In order to use this to our advantage we define the concepts of  Refinement Trees and  \NVB -equivalence, this equivalence defines the same equivalence classes for simplices as in \cite{Stevenson:08}.}
\begin{definition}[Refinement Tree and \NVB -equivalence]
Let $\elm = [z_0,\ldots,z_d]_{t_\elm}$. We execute $d$ uniform \NVB -refinements, which we denote in a graph by setting the refinement edges to be the nodes of the graph and the two children to be the \new{refinement edges of the children}. The root is the initial refinement edge $\overline{z_0z_d}$. We call this graph the Refinement Tree $\RT(\elm)$ and call
a $d$-simplex $\elm$ \NVB -equivalent to a simplex $\elm'$, if $\RT(\elm) = \RT(\elm')$. 
\end{definition}

\begin{example}[Refinement Tree for Type 0]
Figure \ref{fig:edgetree} shows  the refinement tree of the simplex $\elm = [z_0,\ldots,z_d]_0$. One can easily see that this is also the refinement tree of $\elm' = [z_d,z_{d-1},\ldots,z_1,z_0]_0$, so $\elm$ is \NVB -equivalent to $\elm'$.
\begin{figure}[htbp]
\begin{center}
\begin{tikzpicture}[]
\node  (z){$\overline{z_0z_d}$}
  child {node  (a) {$\overline{z_0z_{d-1}}$}
    child {node  (b) {$\overline{z_0z_{d-2}}$}
      child {node {$\vdots$}
        child {node  (d) {$\overline{z_0z_1}$}}
        child {node  (e) {$\overline{z_1z_2}$}}
      } 
      child {node {$\vdots$}}
    }
    child {node (g) {$\overline{z_1z_{d-1}}$}
      child {node {$\vdots$}}
      child {node {$\vdots$}}
    }
  }
  child {node  (j) {$\overline{z_1z_d}$}
    child {node  (k) {$\overline{z_1z_{d-1}}$}
      child {node {$\vdots$}}
      child {node {$\vdots$}}
    }
  child {node  (l) {$\overline{z_2z_d}$}
    child {node {$\vdots$}}
    child {node (c){$\vdots$}
      child {node  (o) {$\overline{z_{d-2}z_{d-1}}$}}
      child {node  (p) {$\overline{z_{d-1}z_d}$}
          child [grow=right] {node (q) {Level $d-1$ } edge from parent[draw=none]
            child [grow=up] {node (r) {$\vdots$} edge from parent[draw=none]
              child [grow=up] {node (s) {Level 2  } edge from parent[draw=none]
                child [grow=up] {node (t) {Level 1 } edge from parent[draw=none]
                  child [grow=up] {node (u) {Level 0} edge from parent[draw=none]}
                }
              }
            }
          }
      }
    }
  }
};
\end{tikzpicture}
\caption{Refinement edges of a $d$-simplex of type 0 under $d$ uniform refinements.}
\label{fig:edgetree}
\end{center}
\end{figure}
\end{example}

Note that in principle one can define refinement trees \new{for arbitrary bisection-based refinement}. So to turn a refinement tree into a simplex ordering, one needs to prove that it is one of the $d+1$ (number of types) refinement trees that relate to an actual \NVB refinement.

As we will often deal with $d$ uniform bisection, we use $\grid_{j}$ with a subindex $j$ for $j$ uniform refinements of $\grid$.
We will use $\grid^{i}$ with a superindex $i$ for lower dimensional subentities as follows.

\begin{definition}[$d$-Skeleton of a triangulation]
Let $\grid$ be a triangulation of a domain $\Omega \subset \R^d$. For $0\leq i \leq d$ we denote by $\grid^i$ the set of all simplices of dimension $i$ that are contained in $\grid = \grid^d$ and call $\grid^i$ the $i$-skeleton of $\grid$. We also define $\grid^i(\elm)$ to be the $i$-skeleton of the triangulation that consists of the single element $\elm$.
\end{definition}

In particular $\grid^0=\vertices$ is the set of all vertices, $\grid^1=\edges$ the set of all edges and $\grid^{d-1}=\faces$ the set of all faces.
Note that $\grid^i_j$ is the $i$-skeleton of the $j$-times refined grid and not the $j$-times refined $i$-skeleton.

\begin{lemma}[Type $0$]
\label{lem:type0}
Let $\elm =[z_0,\ldots,z_d]_0$ be a $d$-simplex of type $t_\elm = 0$ that is bisected $d$ times uniformly.
\begin{itemize}
\item Then every edge $\edge \in \grid^1(\elm)$ gets bisected exactly once and no newly created edge gets bisected.
\item Any sub-simplex $\face \in \grid^{d-1}(\elm)$ is bisected by a valid \NVB uniformly $d-1$ times. It is of type $t_\face=0$ and its ordering coincides with the ordering of $\elm$ up to \NVB -equivalence.
\end{itemize}
\end{lemma}
\begin{proof}
The refinement edges of $\elm$ are depicted in Figure \ref{fig:edgetree}.
The distance of the indices decreases by 1 each level, so any edge $\overline{z_iz_j}, i <j$ is bisected exactly once at level $d-j+i$. Also none of the vertices in this tree are midpoints of former bisections, so no newly created edges are bisected. This proves the first claim.

For the second claim we notice that for the $d$ uniform \NVB -refinements of the simplex $\elm$ the refinement edges of the induced refinement of the face 
\[\face_i = \conv\{z_0,\ldots,z_{i-1},z_{i+1},\ldots,z_d\} \in \grid^{d-1}(\elm)\]  are refinement edges of the Refinement Tree of $\elm$. Moreover the children of the edge $\overline{z_nz_m}$ in the induced refinement tree of $\face_i $ are $\overline{z_nz_{m-1}}$ (or $\overline{z_nz_{m-2}}$ if $m-1=i$) and $\overline{z_{n+1}z_m}$ (or $\overline{z_{n+2}z_{m}}$ if $n+1=i$). From this we can clearly see that  $\face_i$ is \NVB -equivalent to $[z_0,\ldots,z_{i-1},z_{i+1},\ldots,z_d]_0$.
\end{proof}

\begin{remark}
This implies by induction that for $0\leq i \leq d-1$ all elements of $\grid^i(\elm)$ are valid \NVB simplices of type $0$.
\end{remark}

We now proceed to present a variant of Lemma~\ref{lem:type0} for elements of type $t_\elm = 1$.

\begin{lemma}[Type $1$]
\label{lem:type1}
Let $\elm = [z_0,z^*,z_1,\ldots,z_{d-1}]_1$ be a $d$-simplex of type $t_\elm = 1$ that is bisected $d$ times uniformly.
\begin{itemize}
\item Then every edge $\edge \in \grid^1(\elm)$ gets bisected exactly once and no newly created edge gets bisected.
\item Any sub-simplex $\face \in \grid^{d-1}(\elm)$ is of type $t_\face=1$, if it contains $z^*$ and of type $0$ else. The ordering of $\face$ coincides with the ordering of $\elm$ up to \NVB equivalence.
\end{itemize}
\end{lemma}

\begin{proof}

The proof is very similar to the proof of Lemma \ref{lem:type0}.

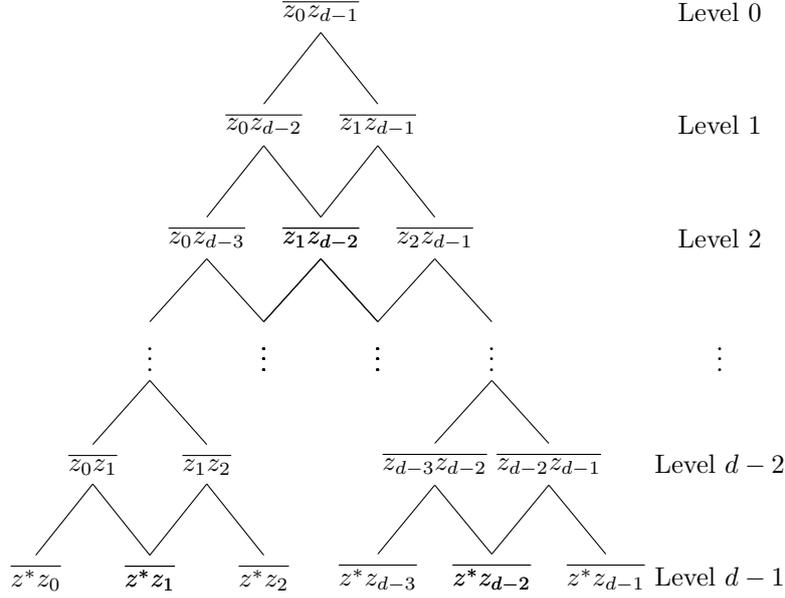
\begin{figure}[htbp]
\begin{center}
\begin{tikzpicture}[]
\node  (z){$\overline{z_0z_{d-1}}$}
  child {node  (a) {$\overline{z_0z_{d-2}}$}
    child {node  (b) {$\overline{z_0z_{d-3}}$}
      child {node {$\vdots$}
        child {node  (d) {$\overline{z_0z_1}$}
          child {node (q) {$\overline{z^*z_0}$}}
          child {node (q) {$\overline{z^*z_1}$}}
        }
        child {node  (e) {$\overline{z_1z_2}$}
          child {node (q) {$\overline{z^*z_1}$}}
          child {node (q) {$\overline{z^*z_2}$}}
        }
      } 
      child {node {$\vdots$}}
    }
    child {node (g) {$\overline{z_1z_{d-2}}$}
      child {node {$\vdots$}}
      child {node {$\vdots$}}
    }
  }
  child {node  (j) {$\overline{z_1z_{d-1}}$}
    child {node  (k) {$\overline{z_1z_{d-2}}$}
      child {node {$\vdots$}}
      child {node {$\vdots$}}
    }
  child {node  (l) {$\overline{z_2z_{d-1}}$}
    child {node {$\vdots$}}
    child {node (c){$\vdots$}
      child {node  (o) {$\overline{z_{d-3}z_{d-2}}$}
          child {node (q) {$\overline{z^*z_{d-3}}$}}
          child {node (q) {$\overline{z^*z_{d-2}}$}}
        }
      child {node  (p) {$\overline{z_{d-2}z_{d-1}}$}
          child {node (q) {$\overline{z^*z_{d-2}}$}}
          child {node (q) {$\overline{z^*z_{d-1}}$}}{
          child [grow=right] {node (q) {Level $d-1$ } edge from parent[draw=none]
            child [grow=up] {node (q) {Level $d-2$ } edge from parent[draw=none]
              child [grow=up] {node (r) {$\vdots$} edge from parent[draw=none]
                child [grow=up] {node (s) {Level 2  } edge from parent[draw=none]
                  child [grow=up] {node (t) {Level 1 } edge from parent[draw=none]
                    child [grow=up] {node (u) {Level 0} edge from parent[draw=none]}
                  }
                }
              }
            }
          }
        }
      }
    }
  }
};
\end{tikzpicture}
\caption{Refinement edges of a $d$-simplex of type 1 under $d$ uniform refinements.}
\label{fig:edgetree1}
\end{center}
\end{figure}
Figure \ref{fig:edgetree1} shows the refinement edges of $\elm$. The distance between the indices decreases by 1 each level until level $d-2$ and that the final level consists of all edges containing $z^*$. So any edge $\overline{z_iz_j}, i < j$ is bisected exactly once at level $d-1-j+i$ and any edge $\overline{z_iz^*}$ is bisected exactly once at level $d-1$. No edge containing a newly created point is bisected, so this proves the first claim.

For the second claim we notice that for the $d$ uniform \NVB -refinements of the simplex $\elm$ the refinement edges of the induced refinement of the face 
\[\face_i = \conv\{z_0,z^*,\ldots,z_{i-1},z_{i+1},\ldots,z_{d-1}\} \in \grid^{d-1}(\elm)\]  are refinement edges of the Refinement Tree of $\elm$. Moreover the children of the edge $\overline{z_nz_m}$ in the induced refinement tree of $\face_i $ are $\overline{z_nz_{m-1}}$ (or $\overline{z_nz_{m-2}}$ if $m-1=i$) and $\overline{z_{n+1}z_m}$ (or $\overline{z_{n+2}z_{m}}$ if $n+1=i$) for $m \neq n+1$. For the case $m=n+1$ the children are  are $\overline{z^*z_n}$ and $\overline{z^*z_m}$. Finally for the case $n=i-1,m=i+1$ the children are $\overline{z^*z_n}$ and $\overline{z^*z_m}$. From this we can clearly see that  $\face_i$ is \NVB -equivalent to $[z_0,z^*,\ldots,z_{i-1},z_{i+1},\ldots,z_d]_1$.

\new{The sub-simplex $\face' = \conv\{z_0,z_1,\ldots,z_{d-1}\}$ is \NVB -equivalent to $[z_0,z_1,\ldots,z_{d-1}]_0$ as every edge $\edge = \overline{z_iz_j}$ gets bisected at level $d-1-j+i$.} This proves the second claim.
\end{proof}

\begin{remark}
We can iterate Lemma \ref{lem:type1} in combination with Lemma \ref{lem:type0} to see that for a $d$ simplex $\elm$ with $t_\elm =1$ any lower dimensional sub-simplex $\face \in \grid^i(\elm), 2\leq i \leq d-1$ it holds $t_{\face} = 1$ if $z^* \in \face$ and $t_\face = 0$ else.
\end{remark}

For types $_\elm=0,1$ we have the convenient result that the number of refinements to bisect every edge is $d$. For types greater than $1$ this is unfortunately not true.

\begin{lemma}[Type $\geq 2$]
\label{lem:type2}
Let $\elm = [z_0,z^*_1,\ldots,z^*_t,z_1,\ldots,z_{d-t}]_t$ be a $d$-simplex of type $t_\elm = t \geq 2$ that is bisected $2d-t$ times uniformly and let $\vertices_i = \{z^*_1,\ldots,z^*_t\}$.
\begin{itemize}
\item Then every edge of $\elm$ gets bisected at least once. 
\item Any face $\face \in \grid^{d-1}(\elm)$ is of type $t_\face = \# (\elm' \cap \vertices_i)$. The ordering of $\face$ coincides with the ordering of $\elm$ up to \NVB equivalence.
\end{itemize} 
\end{lemma}
\begin{proof}
Bisecting the element $\elm = [z_0,z^*_1,\ldots,z^*_t,z_1,\ldots,z_{d-t}]_t$ uniformly $d-t$ times leads to a set of descendants of type 0 that look as follows
\[
\elm' = [\bar{z}_0,\bar{z}_1,\ldots,\bar{z}_{d-t},z^*_1,\ldots,z^*_t]_0
\]
where $\bar{z}_k$ are either vertices of $\elm$ or centers of refinement edges of $\elm$. This means that none of the refinement edges $\edge_{ij}=\overline{z_i^*z_j^*}$, $1<i<j<t$ has been refined yet. From the proof of Lemma \ref{lem:type0} we know that $\edge_{ij}$ will be refined at level $d-j+i$ of $\elm'$ and hence at level $d-t+d-j+i = 2d-t-j+i$ of $\elm$. In particular for $\edge_{12}$ this means $2d-t-1$, so after $2d-t$ uniform bisections every edge in $\elm$ has been refined at least once.

For the second claim we construct an artificial ancestor $\tilde{\elm} = [z_0,\ldots,z_d]_0$ by assuming that the simplex $\elm$ has been created by $t$ bisections always being the first child. From Lemma \ref{lem:type0} we know that the ordering of all subsimplices of $\tilde{\elm}$ coincides with the ordering of $\tilde{\elm}$ and this translates to the descendant. Furthermore $\tilde{\elm}$ has been refined $t$ times introducing the vertices $\vertices_i$, which means that any subsimplex $\tilde{\elm'}$ has been refined as many times as it contains vertices from $\vertices_i$, which proves the second claim.

In particular for $\elm=[z_0,z^*_1,\ldots,z^*_t,z_1,\ldots,z_{d-t}]_t$ any subsimplex $\elm' \in \grid^{d-1}(\elm)$ is of type $t_{\elm'} = \# \{z^*_i | z^*_i \in \elm'\}$ as the subsimplex of the artificial ancestor has been refined as many times.
\end{proof}

\begin{remark}
Lemma \ref{lem:type2} generalizes Remark $3.4$ of~\cite{GaScSt:14}.
\end{remark}

The Lemmata \ref{lem:type0}, \ref{lem:type1} and \ref{lem:type2} lead to the following corollary by simple induction.

\begin{corollary}[Trace Meshes]
For $0<i<d$ the meshes $\grid^i_d$ can be seen as standard \NVB meshes on their own.
\end{corollary}

This means for any simplex $\elm$ with type $ 0 \leq t_\elm \leq d-1$ for $0\leq i \leq d-1$ all elements of $\grid^i(\elm)$ are valid \NVB simplices.

\subsection{Compatibility Condition}

As refinement is defined element-wise with \NVB, we have to define a compatibility condition that couples neighboring elements. The weakest possible condition is the following.
\begin{definition}[Weakest Compatibility Condition]
\label{def:weakestCompCond}
A triangulation $\grid$ fulfils the weakest compatibility condition, iff for any pair of neighbors $\elm_i, \elm_j \in \grid^d$ and their shared face $\face \in \grid^{d-1}$ with $\face = \elm_i \cap \elm_j$ the ordering and type of $\face$ within both elements coincide up to \NVB equivalence.
\end{definition}

This basically means that the induced refinement of a face is the same from both sides. This is called compatibly divisible in \cite{Stevenson:08} or conformingly marked in \cite{ArMuPo:00} (for $d=3$). This condition is necessary, but may not be sufficient for the iterative \NVB algorithm (\cite{Stevenson:08}) to terminate. So we define a slightly stronger condition. 

\begin{definition}[Weak Compatibility Condition]
\label{def:weakCompCond}
A triangulation $\grid$ is called weakly compatible, iff it fulfils condition \ref{def:weakestCompCond} and there exists a finer Triangulation that is conforming and only contains elements of type $t \in \{0,1\}$.
\end{definition}

\begin{remark}
The recursive algorithm \cite{SchmidtSiebert:05} does in general not terminate under these conditions, as neither condition \ref{def:weakestCompCond} nor condition \ref{def:weakCompCond} imply that the grid is loop-free.
\end{remark}

\begin{remark}
Any $2$-dimensional triangulation is weakly compatible, as it only contains elements of type $0$ and $1$.
\end{remark}

\begin{lemma}[Every $d$-th uniform refinement is conforming]
\label{lem:drefconf}
Let $\grid$ be conforming and only contain elements of type $t \in \{0,1\}$. Then every $d$-th uniform \NVB refinement is conforming.
\end{lemma}

\begin{proof}
We split the triangulation $\grid$ into its elements and  investigate \new{each} $\grid(\elm)$. Every triangulation $\grid(\elm)$ is refined $d$ times uniformly. As all elements are type $0$ or $1$, we know from Lemmas \ref{lem:type0} and \ref{lem:type1} that all $d-1$-dimensional subsimplices $\face \in \grid^{d-1}(\elm)$ have been uniformly refined $d-1$ times. Now we reassemble the triangulation $\grid$ from the element-triangulations $\grid(\elm)$. Condition \ref{def:weakestCompCond} ensures that subsimplices that are shared by two elements carry the same refinement structure. So the resulting triangulation is conforming.
\end{proof}

From Lemma \ref{lem:drefconf} follows directly that under Condition \ref{def:weakCompCond} there always is a conforming finer grid, that only contains elements of type $t \in \{0,1\}$, which leads to the following corollary. 

\begin{corollary}
The iterative \NVB refinement algorithm terminates on weakly compatible grids for any set $\marked$ of elements marked for refinement.
\end{corollary}

As we will refer to the standard compatibility condition as strong compatibility condition in the remainder of this article we also introduce the notion of reflected neighbors and the strong compatibility condition.

\begin{definition}[Reflected Neighbors]
Two elements $\elm = \conv\{z_0,\ldots,z_{d-1},u\}$
  and $\elm' = \conv \{z_0,\ldots,z_{d-1},v\}$ are called reflected neighbors, if and only if the types $t_\elm = t_{\elm'}$ and the ordering coincides with $u$ and $v$ being at the same position up to \NVB equivalence.
\end{definition}

\begin{definition}[Local Strong Compatibility]
A \new{face} $\face \in \grid^{d-1}$ is called strongly compatible if and only if the two adjacent elements are reflected neighbors, or their direct children, that are adjacent to $\face$ are reflected neighbors.
\end{definition}

\begin{definition}[Strong Compatibility Condition]
\label{def:stronggrid}
A triangulation $\grid$ is strongly compatible, if all faces $\face \in \grid^{d-1}$ are strongly compatible.
\end{definition}

\begin{remark}
Condition \ref{def:stronggrid} implies that all elements in the grid are of the same type.
\end{remark}

We extend the local strong compatibility to include neighbors that differ by one in type.

\begin{definition}[Local Quasi-strong Compatibility]
\label{def:quasistrong}
Let $\face \in \grid^{d-1}$ have the two adjacent elements $\elm, \elm'$ that differ by one in type with $t_{\elm'} = (t_{\elm} +1 ) \mod d$. Then $\face$ is called quasi-strongly compatible if and only if the child of $\elm$ that contains $\face$ is a reflected neighbor of $t_{\elm'}$.
\end{definition}

\new{This definition mimics the behaviour within adaptively refined strongly compatible grids. All initial elements in these grids are of the same type and neighboring refined elements can only differ by one in generation, neighboring elements can also only differ by one in type (modulo $d$). Such refined elements originating from a strongly compatible grid fulfil Definition \ref{def:quasistrong}. }

\section{Algorithm} \label{sec:algo}

The algorithm we design to actually reach a weakly compatible state is straightforward in any dimension.

\begin{algorithm2e}
\caption{A renumbering algorithm to satisfy Condition \ref{def:weakCompCond} }
\label{algo:generic}
Let $\vertices = \grid^0$ be the vertices of the triangulation $\grid$. Then we divide $\vertices$ into disjoint subsets $\vertices_i \subset \vertices, i \in \{0,1\}, \vertices_0 \cup \vertices_1 = \vertices$ and provide an ordering $>_i$ for each of them.

For all $\elm = [z_0, \ldots, z_n]_{t_\elm} \in \grid$ we set $[z_1,\ldots, z_{t_\elm}] = \elm \cap \vertices_1$ in order of $>_1$, $[z_0,z_{t_\elm +1},\ldots,z_n] = \elm \cap \vertices_0$ in order of $>_0$ and $t_\elm = \# (\elm \cap \vertices_1)\mod d$.

If $\elm \cap \vertices_0 = \emptyset$, we set $t_\elm =0 $ and all vertices are sorted ascendingly from $z_0$ to $z_n$ with respect to $>_1$.

\end{algorithm2e}

\begin{theorem}
The resulting triangulation of Algorithm \ref{algo:generic} is weakly compatible.
\end{theorem}

\begin{proof}
The weakest compatibility follows directly from Lemmas \ref{lem:type0}, \ref{lem:type1}, \ref{lem:type2}.

For the second property of Definition \ref{def:weakCompCond} we refine every simplex individually until it is type 0 without conforming closure. This coincides with refining all edges with both vertices contained in $\vertices_0$ and no other edge for every simplex. Hence the resulting triangulation is conforming and only contains elements of type 0.
\end{proof}

\begin{remark}
Note that the descendants of all faces $\face$ that are quasi-strongly compatible in $\grid$ will be strongly compatible in the constructed descendant.
\end{remark}

\subsection{Variants of Choice of the sets}
Now we show some variants of the choice of sets to see how this actually relates to usual grid construction.

\begin{itemize}
\item[\textbf{OT0}] Only Type 0:\\
All elements are type 0 (i.e. $\vertices_0 = \vertices, \vertices_1 = \emptyset$ ). This may be used to reconstruct a strongly compatible mesh in the standard way.
\item[\textbf{ILE}] Initial Longest Edge:\\
Let $C_{ILE} \in \N$ be a constant threshold. Then we define 
\[\vertices_0 = \{ v \in \vertices : v \text{ is in at least } C_{ILE} \text{ Longest Edges } \} \] and $\vertices_1 = \vertices  \setminus \vertices_0$. This results in refining the initial longest edge first for some simplices.
\item[\textbf{LAE}] Least Adjacent Elements:\\
The idea is to refine opposite of vertices that have least adjacent elements. Let $C_{LAE} \in \N$ be a constant threshold. Then we define
\[\vertices_1 = \{ v \in \vertices : v \text{ is in at most } C_{LAE} \text{ Elements} \} \] 
and $\vertices_0 = \vertices \setminus \vertices_1$. This results in refining opposite of the biggest angle for some simplices.
\end{itemize}

\new{These three set choices  are in no way a thorough classification of the possible set choices but a possible selection of straightforward implementations. Another possibility would be for example to} collect all vertices belonging to Longest edges that are longer than average into $\vertices_0$ or collect the vertices with the maximum volume angle into $\vertices_1$. It is also possible to combine strategies.

\subsection{Variants of the Choice of Ordering}

Up to now, we only constructed the sets $\vertices_{(0,1)}$. We are still free to choose the ordering. Figure \ref{fig:orderings} depicts  different results of the algorithm \ref{algo:generic} in 2d. The ordering in Figure \ref{fig:badOrdering} is lowest in the lower left corner and highest in the upper right and increases row-wise. This naive ordering leads to an expensive conforming closure. On the other hand in Figures \ref{fig:goodordering} and \ref{fig:goodordering2} the orderings are created with alternative methods (SRN and SRN2) and lead to all neighbors being strongly compatible.

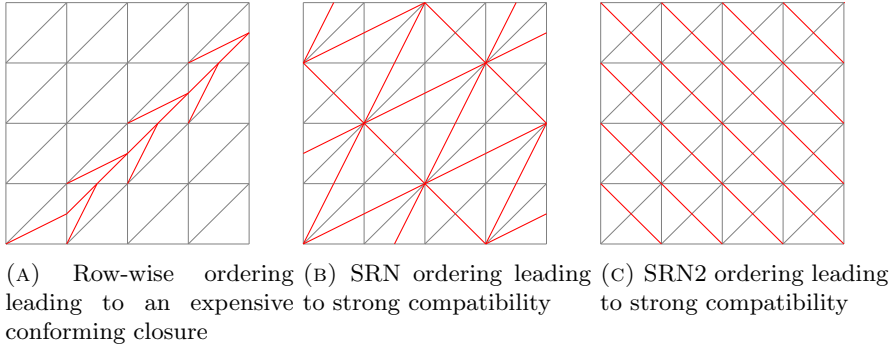
\begin{figure}[t]
  \begin{center}
    \begin{subfigure}[t]{.3\textwidth}
\begin{tikzpicture}[scale = 0.8, every node/.style={scale=0.8}]
\foreach \i [evaluate={\ii=int(\i-1);}] in {0,...,4}{
  \foreach \j [evaluate={\jj=int(\j-1);}] in {0,...,4}{
    \coordinate [shift={(\j,\i)}] (n-\i-\j) at (0,0);
\ifnum\i>0
  \draw [help lines] (n-\i-\j) -- (n-\ii-\j)coordinate[pos=.5](n-\i-\ii-\j-\j);
\fi
\ifnum\j>0
  \draw [help lines] (n-\i-\j) -- (n-\i-\jj)coordinate[pos=.5](n-\i-\i-\j-\jj);
  \ifnum\i>0
    \draw [help lines] (n-\i-\j) -- (n-\ii-\jj)coordinate[pos=.5](n-\i-\ii-\j-\jj);
  \fi
\fi
}}
\foreach \k [evaluate={\kk=int(\k+1); \kkk=int(\k+2);}] in {0,...,3} {
\draw[red] (n-\k-\k) -- (n-\kk-\k-\kk-\kk);
\ifnum\k<3
  \draw[red] (n-\kk-\k-\kk-\kk) -- (n-\kk-\kk-\kkk-\kk);
  \draw[red] (n-\k-\kk) -- (n-\kk-\kk-\kkk-\kk);
  \draw[red] (n-\kkk-\kk-\kkk-\kkk) -- (n-\kk-\kk-\kkk-\kk);
\fi
}
\end{tikzpicture}
\caption{Row-wise ordering leading to an expensive conforming closure}\label{fig:badOrdering}
\end{subfigure}
    \begin{subfigure}[t]{.3\textwidth}
\begin{tikzpicture}[scale=0.8, every node/.style={scale=0.8}]
\foreach \i [evaluate={\ii=int(\i-1);}] in {0,...,4}{
  \foreach \j [evaluate={\jj=int(\j-1);}] in {0,...,4}{
    \coordinate [shift={(\j,\i)}] (n-\i-\j) at (0,0);
\ifnum\i>0
  \draw [help lines] (n-\i-\j) -- (n-\ii-\j)coordinate[pos=.5](n-\i-\ii-\j-\j);
\fi
\ifnum\j>0
  \draw [help lines] (n-\i-\j) -- (n-\i-\jj)coordinate[pos=.5](n-\i-\i-\j-\jj);
  \ifnum\i>0
    \draw [help lines] (n-\i-\j) -- (n-\ii-\jj)coordinate[pos=.5](n-\i-\ii-\j-\jj);
  \fi
\fi
}}

\draw[red] (n-0-0) -- (n-1-0-1-1);
\draw[red] (n-0-0) -- (n-1-1-1-0);

\draw[red] (n-2-1) -- (n-1-1-1-0);
\draw[red] (n-2-1) -- (n-2-1-0-0);
\draw[red] (n-2-1) -- (n-3-2-1-0);
\draw[red] (n-2-1) -- (n-3-3-2-1);
\draw[red] (n-2-1) -- (n-3-2-2-2);
\draw[red] (n-2-1) -- (n-2-1-2-1);

\draw[red] (n-1-2) -- (n-1-0-1-1);
\draw[red] (n-1-2) -- (n-0-0-2-1);
\draw[red] (n-1-2) -- (n-1-0-3-2);
\draw[red] (n-1-2) -- (n-2-1-3-3);
\draw[red] (n-1-2) -- (n-2-2-3-2);
\draw[red] (n-1-2) -- (n-2-1-2-1);

\draw[red] (n-3-3) -- (n-4-3-3-2);
\draw[red] (n-3-3) -- (n-3-2-4-3);
\draw[red] (n-3-3) -- (n-4-3-4-4);
\draw[red] (n-3-3) -- (n-4-4-4-3);
\draw[red] (n-3-3) -- (n-2-2-3-2);
\draw[red] (n-3-3) -- (n-3-2-2-2);

\draw[red] (n-3-0) -- (n-4-4-1-0);
\draw[red] (n-3-0) -- (n-4-3-1-1);
\draw[red] (n-3-0) -- (n-3-2-1-0);

\draw[red] (n-0-3) -- (n-1-0-4-4);
\draw[red] (n-0-3) -- (n-1-1-4-3);
\draw[red] (n-0-3) -- (n-1-0-3-2);

\draw[red] (n-2-4) -- (n-1-1-4-3);
\draw[red] (n-2-4) -- (n-3-2-4-3);
\draw[red] (n-2-4) -- (n-2-1-3-3);
\draw[red] (n-4-2) -- (n-4-3-1-1);
\draw[red] (n-4-2) -- (n-4-3-3-2);
\draw[red] (n-4-2) -- (n-3-3-2-1);
\end{tikzpicture}
\caption{SRN ordering leading to strong compatibility}\label{fig:goodordering}
\end{subfigure}
    \begin{subfigure}[t]{.3\textwidth}
\begin{tikzpicture}[scale=0.8, every node/.style={scale=0.8}]
\foreach \i [evaluate={\ii=int(\i-1);}] in {0,...,4}{
  \foreach \j [evaluate={\jj=int(\j-1);}] in {0,...,4}{
    \coordinate [shift={(\j,\i)}] (n-\i-\j) at (0,0);
\ifnum\i>0
  \draw [help lines] (n-\i-\j) -- (n-\ii-\j)coordinate[pos=.5](n-\i-\ii-\j-\j);
\fi
\ifnum\j>0
  \draw [help lines] (n-\i-\j) -- (n-\i-\jj)coordinate[pos=.5](n-\i-\i-\j-\jj);
  \ifnum\i>0
    \draw [help lines] (n-\i-\j) -- (n-\ii-\jj)coordinate[pos=.5](n-\i-\ii-\j-\jj);
  \fi
\fi
}}
\tikzstyle{refinement} = [red]
\foreach \k in {1,2,3,4}{
  \draw[refinement] (n-0-\k) -- (n-\k-0);
  \draw[refinement] (n-\k-4) -- (n-4-\k);
}
\end{tikzpicture}
\caption{SRN2 ordering leading to strong compatibility}\label{fig:goodordering2}
\end{subfigure}
\caption{The conforming closure to different orderings for a 2d Kuhn-tesselation.}
\label{fig:orderings}
  \end{center}
\end{figure}

We propose two strategies to construct an ordering, the second one being an improved version of the first one.

\begin{itemize}
  \item[\textbf{SRN}] Successive Reflected Neighbors:\\
Build a global ordering that includes all vertices in $\vertices$ successively. Start with an element $\elm = [z_0,\ldots,z_n]_{t_\elm}$ and add its ordering to the global ordering. Then loop over all neighbors and, if the new vertex has not been added to the ordering, insert the new vertex between the replaced vertex and its follower. Then loop over the neighbors of the neighbors and so on.
\item[\textbf{SRN2}] Successive Reflected Neighbors with announced refinement edge:\\
We assume that every simplex announces an edge, which it wants to bisect. In our implementation this is the longest edge, but in principle any edge may be announced.
We slightly alter the Successive Reflected Neigbors strategy.
If the shared face with the neighbor does not contain the refinement edge, we have two options: either insert the new vertex before the first element or after the last. If one of the insertions yields the announced refinement edge, we choose this one.
\end{itemize}

SRN and SRN2 aim at making neighbors of different type (quasi-)strongly compatible, as types can at most differ by $1$ by construction and they create (quasi-)strongly compatible faces opposite of every newly inserted vertex.

A possible implementation of OT0 and SRN2 is shown in Algorithm \ref{algo:possible}. Using a structured cube grid filled with Kuhn-cubes, where every simplex announces the diagonal initially, Algorithm \ref{algo:generic} leads to the standard strongly compatible grid, if the initial element is sorted correctly.

\begin{algorithm2e}
\caption{Possible implementation of methods OT0 and SRN2}
\label{algo:possible}
\KwData{List of active faces $\freefaces$, List of vertices $\vertices$, Mesh $\grid$}
$\freefaces  = \vertices = \emptyset$\\
Choose initial $\elm \in \grid$ \\
Order according to announced refinement edge $\edge(\elm)$\\
$\vertices = \vertices (\elm)$\\
add all non-boundary $\face \in \grid^{d-1}(\elm)$ with $\edge(\elm) \in \face$ at the beginning of $\freefaces$ \\
add the other non-boundary $\face \in \grid^{d-1}(\elm)$ at the end of $\freefaces$ with the flag noRefEdge\\
mark $\elm$ as treated\\
\While{ $\exists$ untreated $\elm \in \grid$ ($\Leftrightarrow \freefaces \neq \emptyset$)}{
Get untreated neighbor $\elm'$ and treated element $\elm$ of first $\face \in \freefaces$\\ 
$v = \grid^{0}(\elm) \setminus \grid^0(\face) \quad v' = \grid^{0}(\elm') \setminus \grid^0(\face)$\\
\eIf{$v' \in \vertices$}
{
  \tcc{do nothing}
}
{  \eIf {noRefEdge set}{
  \tcc{In this case $v$ is the first or last vertex of $\elm$}
insert $v'$ after last or before first vertex of $\elm$, depending on the announced refinement edge\\
}{
Insert $v'$ directly after $v$ into $\vertices$\\
}
}
sort $\elm'$ according to $\vertices$\\
mark $\elm'$ treated\\
\For{non-boundary $\face' \in \grid^{d-1}(\elm)$}{
\eIf{$\face' \in \freefaces$}{
  \tcc{possibly check for strong compatibility}
remove $\face'$ from $\freefaces$\\
}{
\eIf{$\edge(\elm') \subset \face'$ }{add $\face'$ at beginning of $\freefaces$}{add $\face$ at the end of $\freefaces$}
}

}
}

\end{algorithm2e}

\subsection{Quality of produced mesh}

\new{We are interested in two significant mesh quality goals, namely, the \emph{mesh topological  quality} goal that tries to minimize the effort of the conforming closure,
 and the \emph{geometric quality} goal that tries to maximize the shape regularity of the elements.
 
  Strongly compatible grids fulfil the \emph{mesh topological  quality} goal by design.} So a distance to a strongly compatible situation is a way of measuring this goal. 
To this aim we define the following two distances.

\begin{definition}[Distance 1]
Let $\grid$ be a weakly compatible grid sorted by one of the above methods and let $\faces_{SC} \subset \grid^{d-1}$ be the set of faces that are strongly or quasi-strongly compatible. Then we define the distance 
\[
d^1_\grid = \#(\grid^{d-1}) - \#(\faces_{SC}).
\]
\end{definition}

This is a simple,\new{ easily computable} and straight-forward metric that directly tells us whether the grid fulfils the strong compatibility condition. Unfortunately, it does not directly yield information on the effort of the conforming closure, \new{so} to estimate this we define a second distance

\begin{definition}[Distance 2]
Let $\grid$ be a weakly compatible grid sorted by one of the above variants and let $\mathcal{C}_\grid(\elm)$ the conforming closure of refining $\elm$ on the initial grid $\grid$. Then we define the metric 
\[
d^2_\grid = \max_{\elm \in \grid} \#(\mathcal{C}_\grid(\elm)).
\]
\end{definition}

As the refinement becomes more local (cf. \cite{AlkamperKlofkorn:17}) once the whole grid has been refined $d$ times uniformly, we also investigate $d^2_{\grid_d}$ on the $d$ times uniformly refined grid $\grid_d$.

\new{The geometric quality of meshes using bisection (both LEB and \NVB) is always an issue. One of the angles may be divided $2^{d-1}$ times under $d$ uniform bisections. LEB performs a bit better by design, but there is a significant drop in quality for any bisection-based refinement, if the initial grid contains elements that are close to equilateral, which we expect from mesh-generators such as \texttt{TetGen}. In principle it is advisable to start with elements that are similar to Kuhn-simplices. For \NVB we have, that $d$ uniform bisections cover all possible similarity classes and hence will only measure our quality indicators on the initial grid and on the $d$ times uniform refined grid.

The set of indicators we use } are the $d$-dimensional sine, aspect ratios of volumes, faces and edges (min, max and average) and also the maximum number of adjacent elements of vertices and edges.

\section{Numerical Experiments} \label{sec:numerics}

In this section we study the behavior of the presented options 
of the reordering algorithm for a variety of different tetrahedral grids. 
In the implementation first a check is performed whether the provided grids is compatible or not. If the grid is already compatible then no reordering is performed. In general the reordering has to be performed only once per grid. 
The implementation of the presented algorithm is contained in the 
open-source \dune module \dunealugrid \cite{AlDeKlNo:16}.

\subsection{Threshold Study}
The first set of experiments we conducted is a threshold study, where we measure the behaviour of LAE and ILE with threshold values $C_{ILE}, C_{LAE} \in \{0,\ldots,35\}$. If the threshold value is $0$ we obtain in both cases $\vertices_0 = \vertices$, so the case OT0 is implicitly included. We use both suggested orderings SRN and SRN2.

First we test on a sequence of triangulations of the unit cube with decreasing average volume,
that have been used in the
\textit{3D Benchmark on Discretization Schemes for Anisotropic Diffusion
Problems on General Grids} \cite{fvca6_bench:11}. \new{We will call these grids the "Unitcube Sequence".}

\begin{figure}[!ht]
  \begin{center}
  \includegraphics[width=0.3\textwidth]{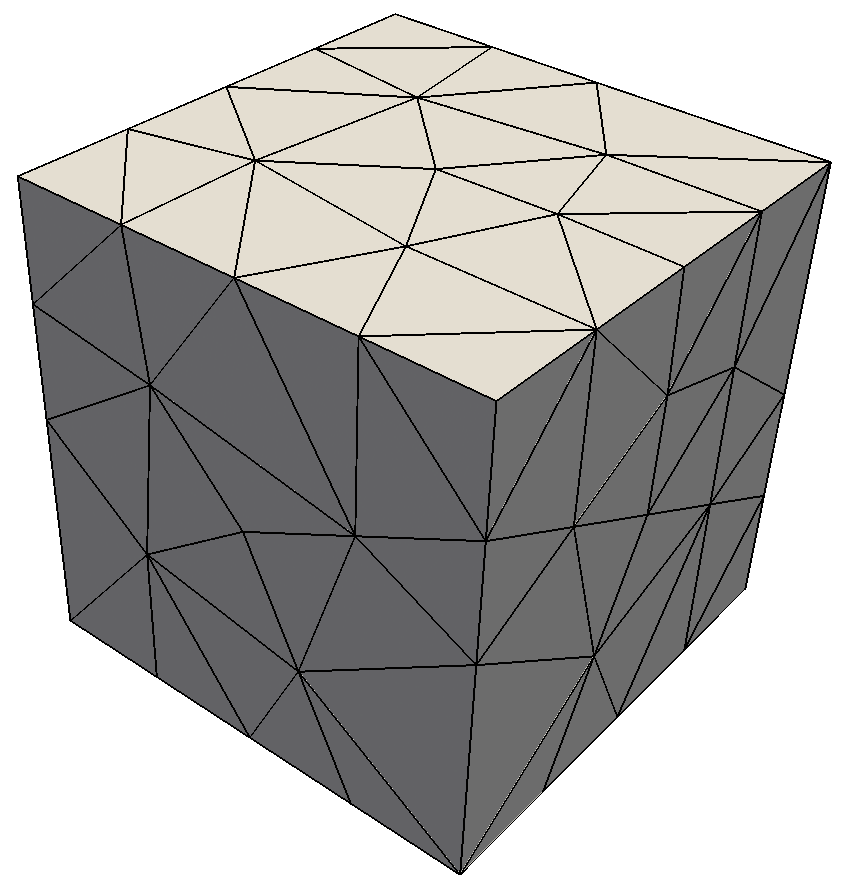}
  \includegraphics[width=0.3\textwidth]{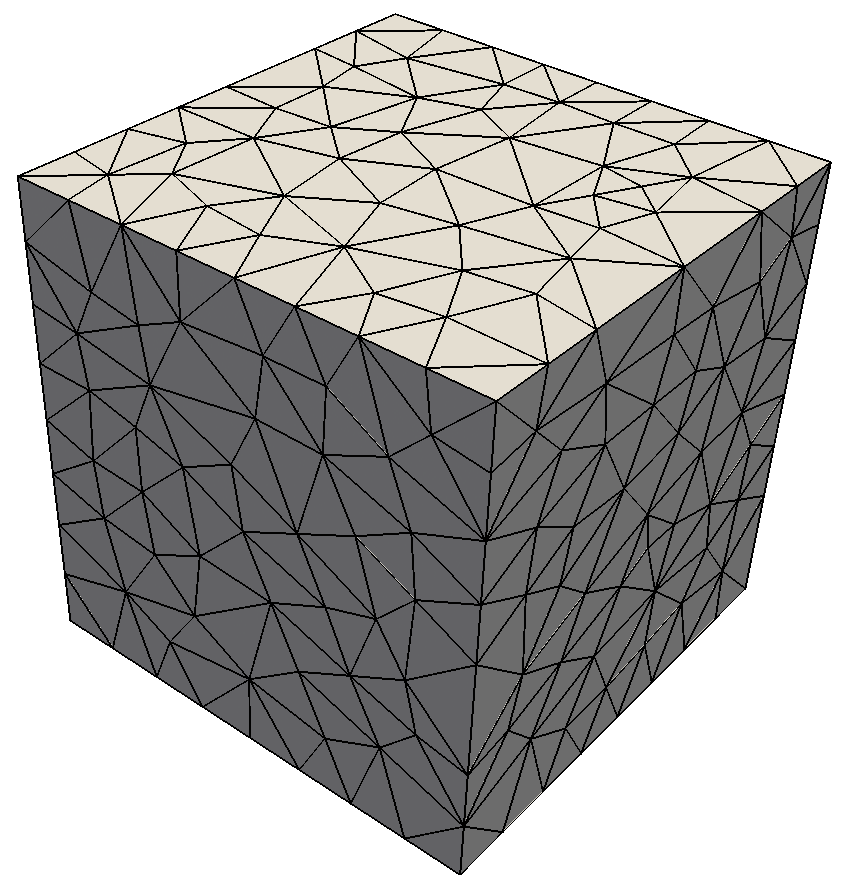}
  \includegraphics[width=0.3\textwidth]{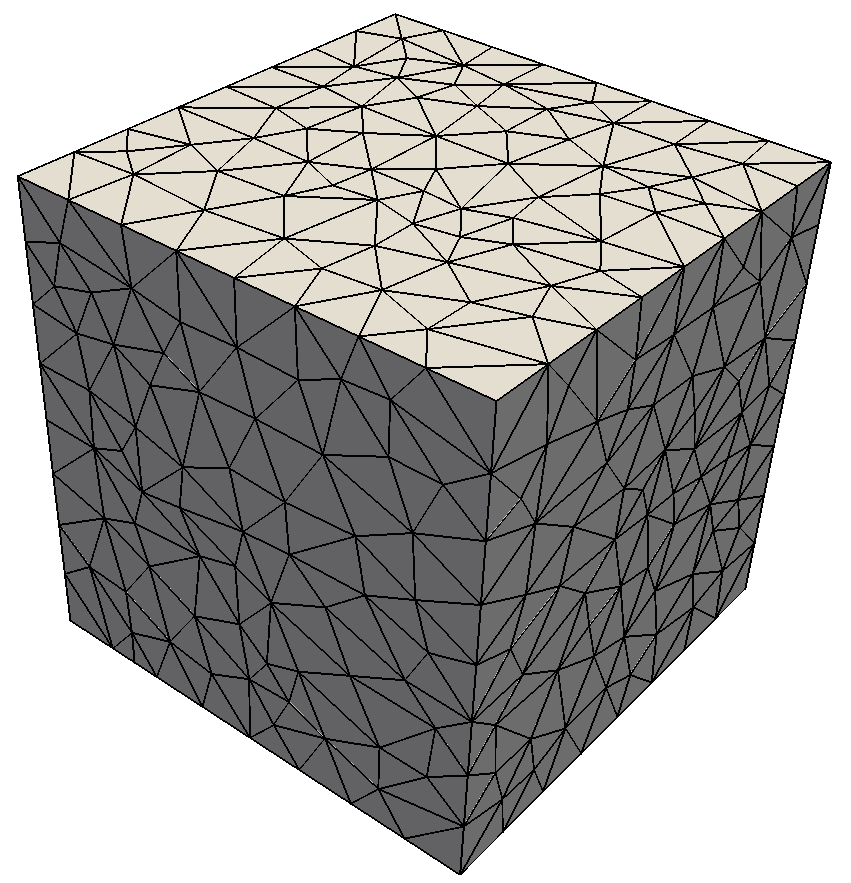}
  \caption{Series of tetrahedral grids discretizing the unit cube used in the 
    \textit{3D Benchmark on Discretization Schemes for Anisotropic Diffusion 
Problems on General Grids} \cite{fvca6_bench:11}.}
  \label{fig:unitcubes}
  \end{center}
\end{figure}

Additionally we will show the same measures for a set of grids representing more
complex geometrical structures. \new{We will call these grids "Realistic Grids" and a
depiction is presented in Figure \ref{fig:realgrids}.}
The tube grid has been previously used for numerical simulation of atherosclerotic plaque formation 
\cite{blood:14}.
The grid representing the head of a human has been used in \cite{head:15}.

\begin{figure}[!ht]
  \begin{center}
  \begin{subfigure}{0.45\textwidth}
  \includegraphics[width=\linewidth]{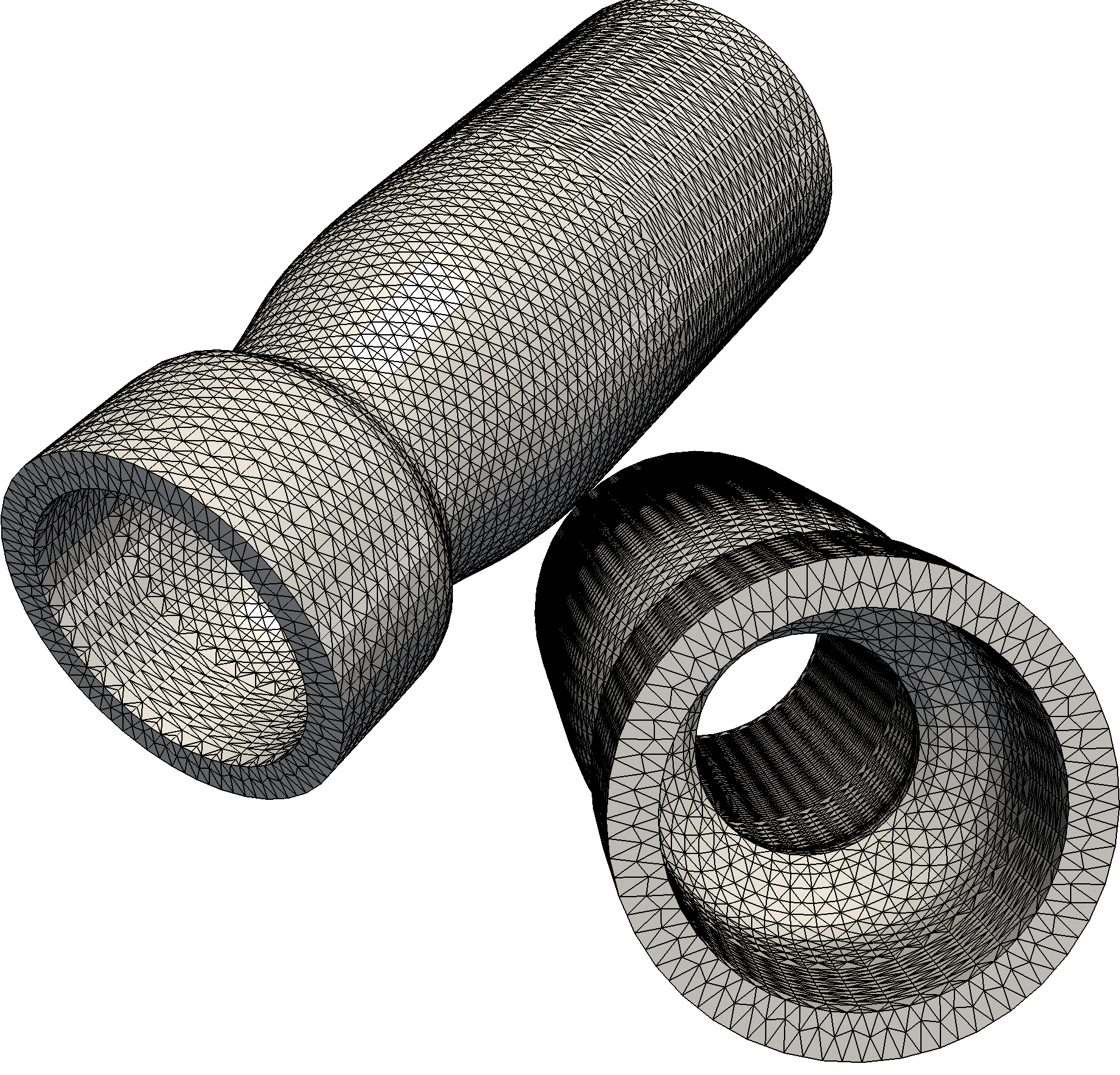}
  \caption{tube}
  \label{fig:tubegrid}
  \end{subfigure}
  \begin{subfigure}{0.45\textwidth}
  \includegraphics[width=\linewidth]{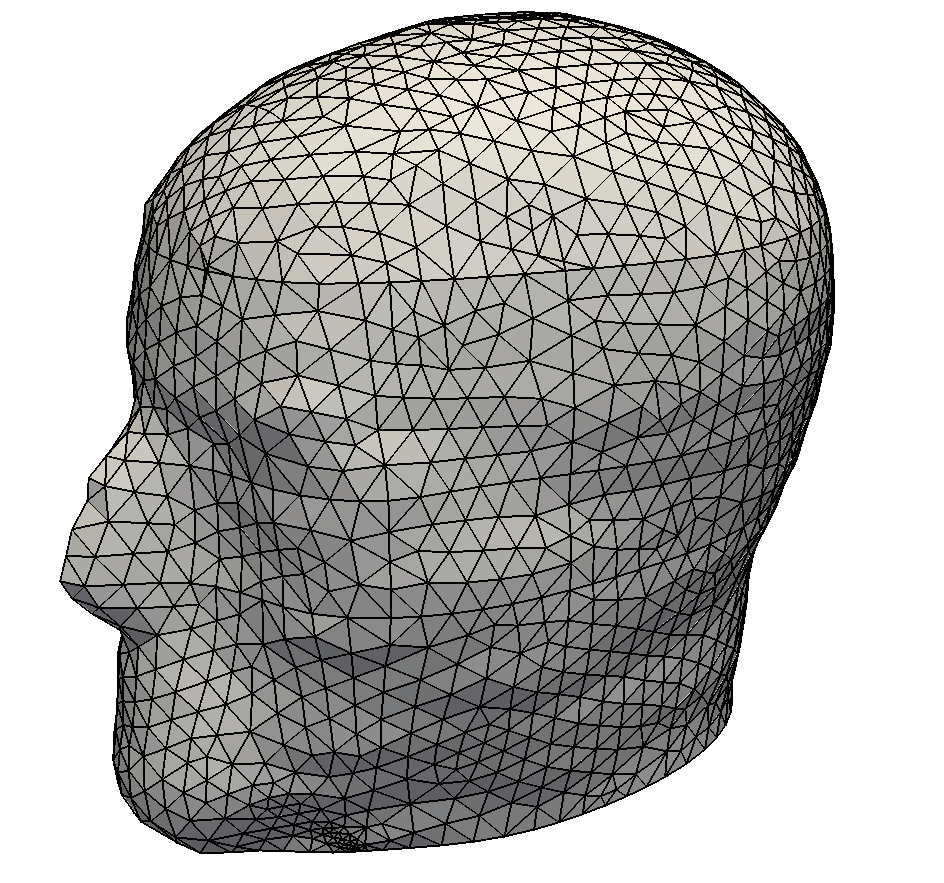}
    \caption{head}
  \label{fig:headgrid}
  \end{subfigure} \\
  \begin{subfigure}{0.4\textwidth}
  \includegraphics[width=\linewidth]{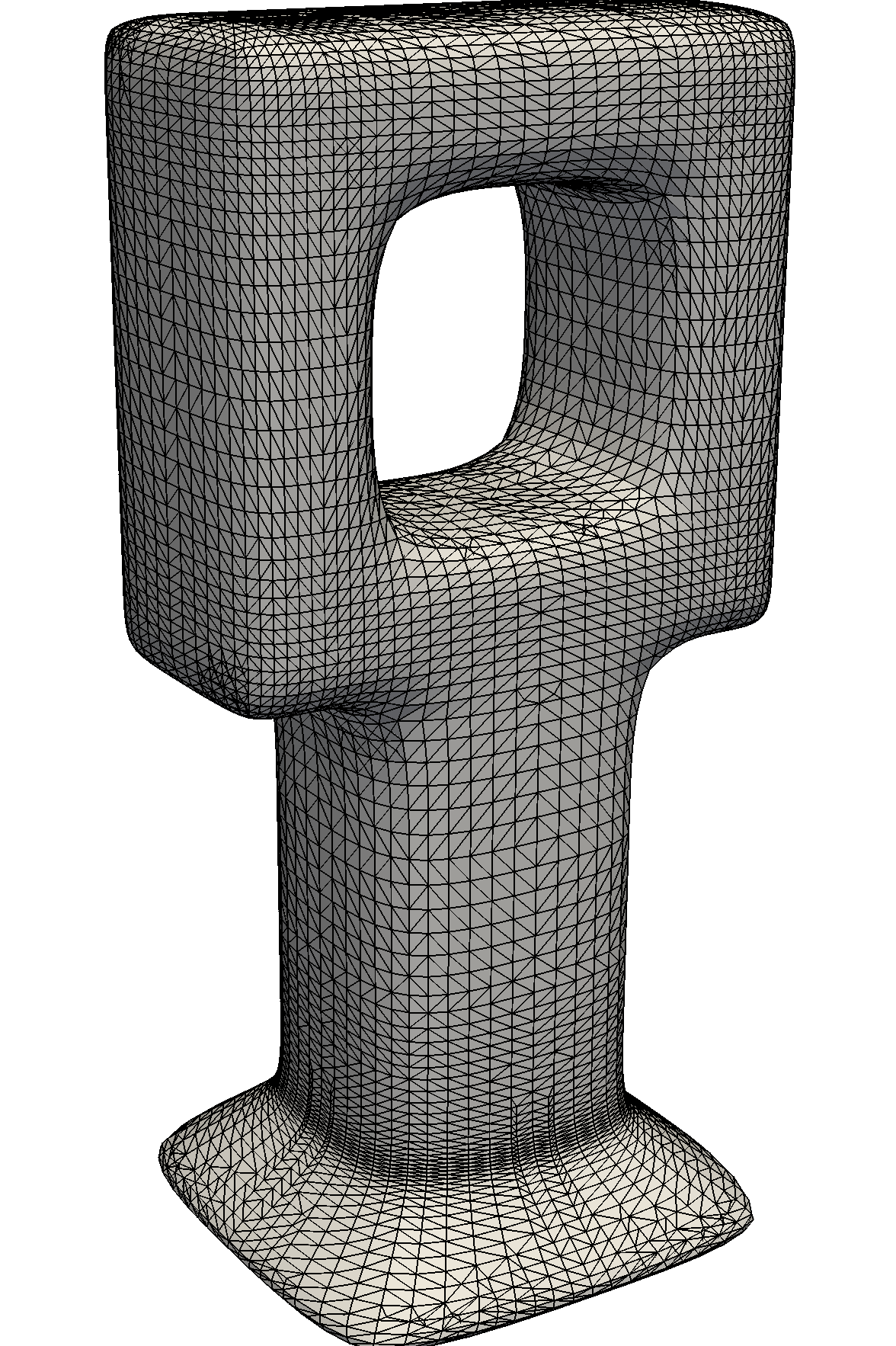}
    \caption{decorstatue}
    \label{fig:decorstatue}
  \end{subfigure}
    \begin{minipage}{0.55\textwidth}
  \begin{subfigure}{\textwidth}
    \begin{center}
  \includegraphics[width=0.9\textwidth]{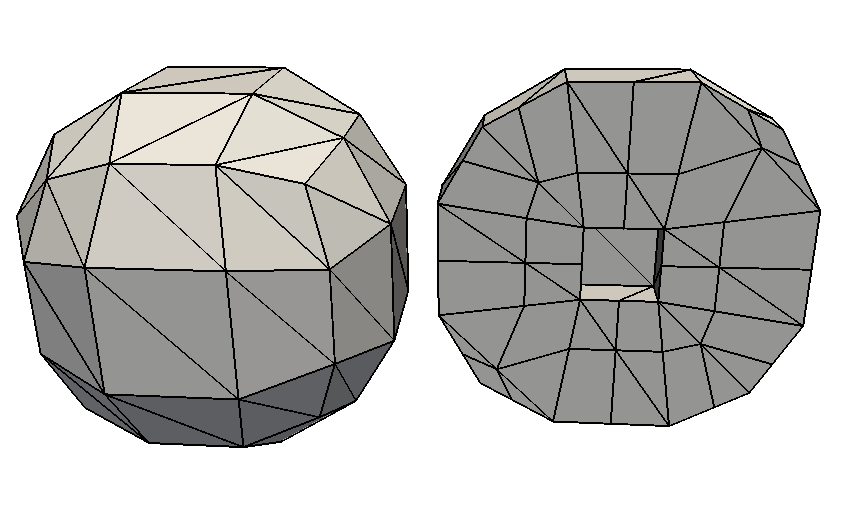} 
      \vskip -5pt
      \caption{Kuhn-grid}
      \label{fig:kuhngrid}
    \end{center}
  \end{subfigure} \ \\ \vskip 5pt
  \begin{subfigure}{0.49\textwidth}
    \begin{center}
  \includegraphics[width=0.9\textwidth]{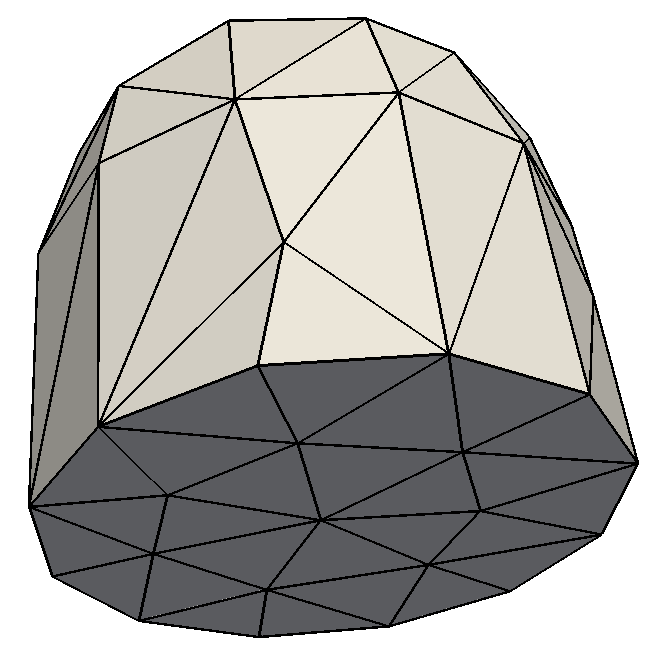} 
      \caption{telescope}
      \label{fig:telescope}
    \end{center}
  \end{subfigure}
  \begin{subfigure}{0.49\textwidth}
    \begin{center}
  \includegraphics[width=0.9\textwidth]{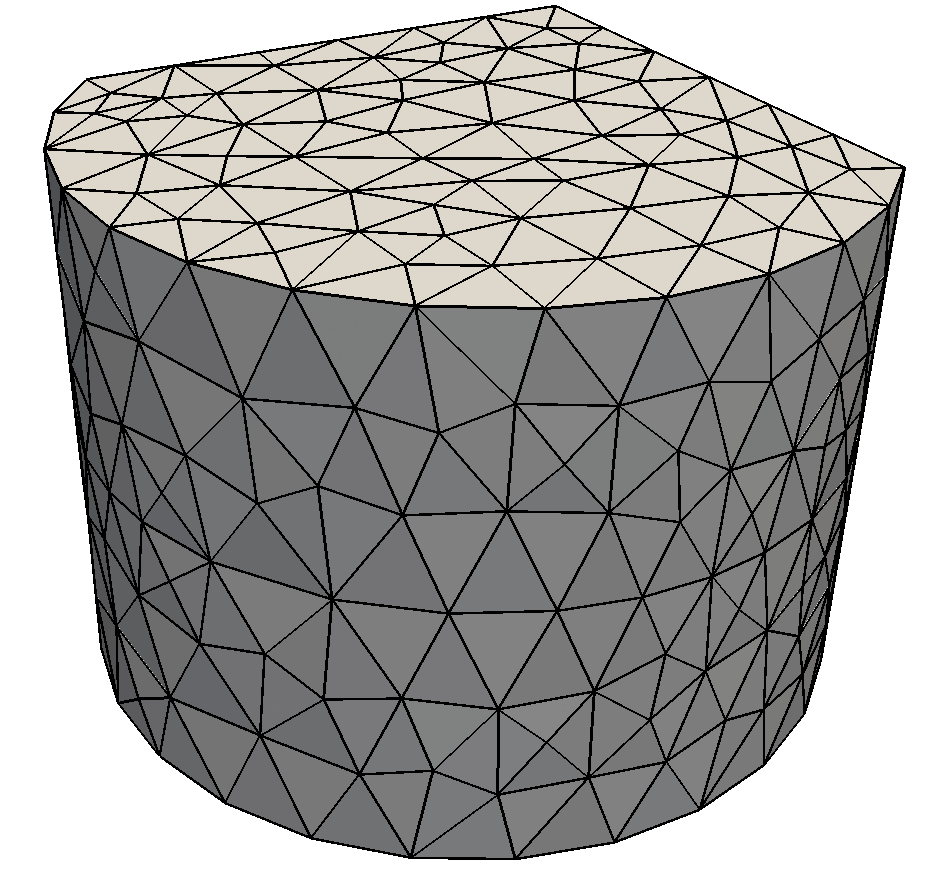} 
      \caption{half circle}
      \label{fig:halfcircle}
    \end{center}
  \end{subfigure}
  \end{minipage}
  \caption{Realistic tetrahedral meshes used for studying the proposed
    algorithms. \ref{fig:tubegrid} has been previously 
    used for numerical simulation of atherosclerotic plaque formation \cite{blood:14}, 
    \ref{fig:headgrid} has been used in \cite{head:15} for electrical impedance tomography simulations, 
    \ref{fig:decorstatue} has been downloaded from \cite{inriapage} and
    generated with \tetgen, \ref{fig:telescope} is part of the test grids in the \dunegrid module \cite{dune24:16}, 
    \ref{fig:kuhngrid} stems from a Kuhn-grid and has
    been modified (vertex renumbering and projection), and \ref{fig:halfcircle} has been generated with \gmsh.}
  \label{fig:realgrids}
  \end{center}
\end{figure}

First we consider the distribution of $\vertices_0,\vertices_1$:
\begin{figure}[ht]
\begin{subfigure}{0.49\textwidth}
\begin{tikzpicture}[scale=0.405]
\begin{axis}[xlabel={$C_{ILE}$}]
\foreach \i in {1,...,6}{
\addplot table[ x expr=\coordindex, y expr= \thisrowno{0} / (\thisrowno{1} + \thisrowno{0}), col sep=space, mark={} ] {./tet.\i.msh.V0V1-1-0.plot};
}
\end{axis}
\end{tikzpicture}
\begin{tikzpicture}[scale=0.405]
\begin{axis}[xlabel={$C_{LAE}$}]
\foreach \i in {1,...,6}{
\addplot table[ x expr=\coordindex, y expr= \thisrowno{0} / (\thisrowno{1} + \thisrowno{0}), col sep=space, mark={} ] {./tet.\i.msh.V0V1-2-0.plot};
}
\end{axis}
\end{tikzpicture}
\caption{Unitcube Sequence}
\end{subfigure}
\begin{subfigure}{0.49\textwidth}
\begin{tikzpicture}[scale=0.405]
\begin{axis}[xlabel={$C_{ILE}$}]
\addplot table[ x expr=\coordindex, y expr= \thisrowno{0} / (\thisrowno{1} + \thisrowno{0}), col sep=space ] {./telescope2ndorder.V0V1-1-0.plot};
\addlegendentry{telescope};
\addplot table[ x expr=\coordindex, y expr= \thisrowno{0} / (\thisrowno{1} + \thisrowno{0}), col sep=space ] {./tube3d.V0V1-1-0.plot};
\addlegendentry{tube};
\addplot table[ x expr=\coordindex, y expr= \thisrowno{0} / (\thisrowno{1} + \thisrowno{0}), col sep=space ] {./TA052.V0V1-1-0.plot};
\addlegendentry{head};
\addplot table[ x expr=\coordindex, y expr= \thisrowno{0} / (\thisrowno{1} + \thisrowno{0}), col sep=space ] {./decorstatue.V0V1-1-0.plot};
\addlegendentry{decorstatue};
\addplot table[ x expr=\coordindex, y expr= \thisrowno{0} / (\thisrowno{1} + \thisrowno{0}), col sep=space ] {./halfcircle.V0V1-1-0.plot};
\addlegendentry{halfcircle};
\addplot table[ x expr=\coordindex, y expr= \thisrowno{0} / (\thisrowno{1} + \thisrowno{0}), col sep=space ] {./hole3.V0V1-1-0.plot};
\addlegendentry{Kuhn-grid};
\end{axis}
\end{tikzpicture}
\begin{tikzpicture}[scale=0.405]
\begin{axis}[legend style={at={(0.03,0.03)},anchor=south west},xlabel={$C_{LAE}$}]
\addplot table[ x expr=\coordindex, y expr= \thisrowno{0} / (\thisrowno{1} + \thisrowno{0}), col sep=space ] {./telescope2ndorder.V0V1-2-0.plot};
\addplot table[ x expr=\coordindex, y expr= \thisrowno{0} / (\thisrowno{1} + \thisrowno{0}), col sep=space ] {./tube3d.V0V1-2-0.plot};
\addplot table[ x expr=\coordindex, y expr= \thisrowno{0} / (\thisrowno{1} + \thisrowno{0}), col sep=space ] {./TA052.V0V1-2-0.plot};
\addplot table[ x expr=\coordindex, y expr= \thisrowno{0} / (\thisrowno{1} + \thisrowno{0}), col sep=space ] {./decorstatue.V0V1-2-0.plot};
\addplot table[ x expr=\coordindex, y expr= \thisrowno{0} / (\thisrowno{1} + \thisrowno{0}), col sep=space ] {./halfcircle.V0V1-2-0.plot};
\addplot table[ x expr=\coordindex, y expr= \thisrowno{0} / (\thisrowno{1} + \thisrowno{0}), col sep=space ] {./hole3.V0V1-2-0.plot};
\end{axis}
\end{tikzpicture}
\caption{Realistic grids}
\end{subfigure}
\caption{The proportion of vertices sorted into $\vertices_0$ over different thresholds.}
\label{fig:V0V1}
\end{figure}
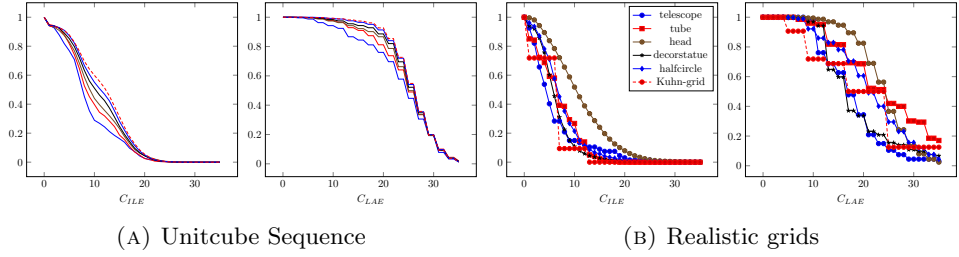

Figure \ref{fig:V0V1} shows the proportion of vertices that have been sorted into $\vertices_0$ for different grids. As expected LAE and ILE have different "active" regions, i.e. regions, where a change of threshold changes the distribution a lot. For ILE this is $C_{ILE} \in \{0,\ldots,20\}$ and for $LAE$ this is $C_{LAE} \in \{10,\ldots,30\}$. Keep these active regions in mind, when examining the following figures. The distribution of $\vertices$ into $\vertices_0$ and $\vertices_1$ does not depend on the choice of ordering. So the active regions do not change from SRN to SRN2.

We now investigate metric $d^1_\grid$, i.e. the amount of not strongly compatible faces. To be comparable \new{between grids of different size}, we display the percentage with respect to all inner faces. Boundary faces are excluded, as they do not need to be compatible.

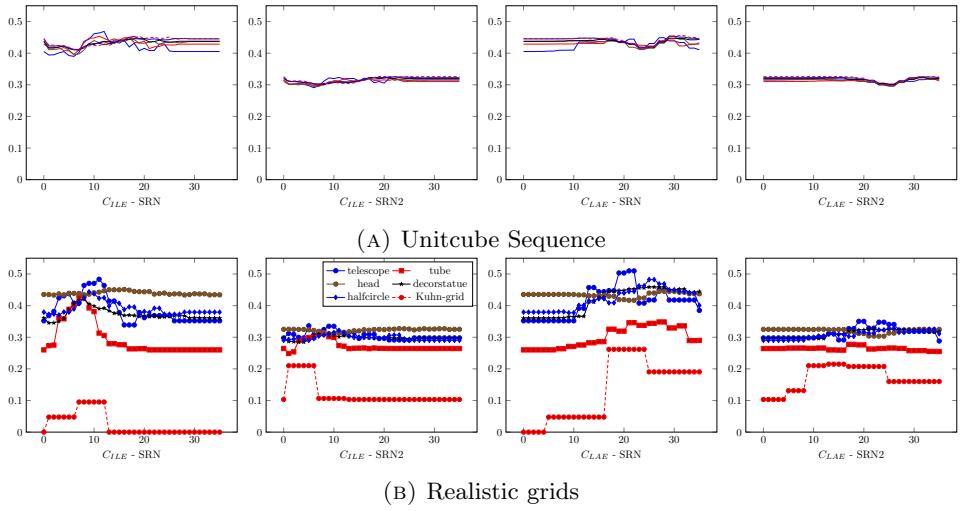
\begin{figure}[ht]
\begin{subfigure}{0.99\textwidth}
\begin{tikzpicture}[scale=0.405]
\begin{axis}[xlabel={$C_{ILE}$ - SRN}, ymin=0, ymax=0.55]
\foreach \i in {1,...,6}{
\addplot table[ x expr=\coordindex, y expr= \thisrowno{0} / \thisrowno{1}, col sep=space, mark={}] {./tet.\i.msh.NonCompat-1-0.plot};
}
\end{axis}
\end{tikzpicture}
\begin{tikzpicture}[scale=0.405]
\begin{axis}[xlabel={$C_{ILE}$ - SRN2}, ymin=0, ymax=0.55]
\foreach \i in {1,...,6}{
\addplot table[ x expr=\coordindex, y expr= \thisrowno{0} / (\thisrowno{1} + \thisrowno{0}), col sep=space, mark={}] {./tet.\i.msh.NonCompat-1-1.plot};
}
\end{axis}
\end{tikzpicture}
\begin{tikzpicture}[scale=0.405]
\begin{axis}[xlabel={$C_{LAE}$ - SRN}, ymin=0, ymax=0.55]
\foreach \i in {1,...,6}{
\addplot table[ x expr=\coordindex, y expr= \thisrowno{0} / \thisrowno{1}, col sep=space, mark={}] {./tet.\i.msh.NonCompat-2-0.plot};
}
\end{axis}
\end{tikzpicture}
\begin{tikzpicture}[scale=0.405]
\begin{axis}[xlabel={$C_{LAE}$ - SRN2}, ymin=0, ymax=0.55]
\foreach \i in {1,...,6}{
\addplot table[ x expr=\coordindex, y expr= \thisrowno{0} / (\thisrowno{1} + \thisrowno{0}), col sep=space, mark={}] {./tet.\i.msh.NonCompat-2-1.plot};
}
\end{axis}
\end{tikzpicture}
\caption{Unitcube Sequence}
\end{subfigure}
\begin{subfigure}{0.99\textwidth}
\begin{tikzpicture}[scale=0.405]
\begin{axis}[xlabel={$C_{ILE}$ - SRN}, ymin=0, ymax=0.55]
\addplot table[ x expr=\coordindex, y expr= \thisrowno{0} / \thisrowno{1}, col sep=space ] {./telescope2ndorder.NonCompat-1-0.plot};
\addplot table[ x expr=\coordindex, y expr= \thisrowno{0} / \thisrowno{1}, col sep=space ] {./tube3d.NonCompat-1-0.plot};
\addplot table[ x expr=\coordindex, y expr= \thisrowno{0} / \thisrowno{1}, col sep=space ] {./TA052.NonCompat-1-0.plot};
\addplot table[ x expr=\coordindex, y expr= \thisrowno{0} / \thisrowno{1}, col sep=space ] {./decorstatue.NonCompat-1-0.plot};
\addplot table[ x expr=\coordindex, y expr= \thisrowno{0} / \thisrowno{1}, col sep=space ] {./halfcircle.NonCompat-1-0.plot};
\addplot table[ x expr=\coordindex, y expr= \thisrowno{0} / \thisrowno{1}, col sep=space ] {./hole3.NonCompat-1-0.plot};
\end{axis}
\end{tikzpicture}
\begin{tikzpicture}[scale=0.405]
\begin{axis}[legend columns=2, xlabel={$C_{ILE}$ - SRN2}, ymin=0, ymax=0.55]
\addplot table[ x expr=\coordindex, y expr= \thisrowno{0} / (\thisrowno{1} + \thisrowno{0}), col sep=space ] {./telescope2ndorder.NonCompat-1-1.plot};
\addlegendentry{telescope};
\addplot table[ x expr=\coordindex, y expr= \thisrowno{0} / (\thisrowno{1} + \thisrowno{0}), col sep=space ] {./tube3d.NonCompat-1-1.plot};
\addlegendentry{tube};
\addplot table[ x expr=\coordindex, y expr= \thisrowno{0} / (\thisrowno{1} + \thisrowno{0}), col sep=space ] {./TA052.NonCompat-1-1.plot};
\addlegendentry{head};
\addplot table[ x expr=\coordindex, y expr= \thisrowno{0} / (\thisrowno{1} + \thisrowno{0}), col sep=space ] {./decorstatue.NonCompat-1-1.plot};
\addlegendentry{decorstatue};
\addplot table[ x expr=\coordindex, y expr= \thisrowno{0} / (\thisrowno{1} + \thisrowno{0}), col sep=space ] {./halfcircle.NonCompat-1-1.plot};
\addlegendentry{halfcircle};
\addplot table[ x expr=\coordindex, y expr= \thisrowno{0} / (\thisrowno{1} + \thisrowno{0}), col sep=space ] {./hole3.NonCompat-1-1.plot};
\addlegendentry{Kuhn-grid};
\end{axis}
\end{tikzpicture}
\begin{tikzpicture}[scale=0.405]
\begin{axis}[xlabel={$C_{LAE}$ - SRN}, ymin=0, ymax=0.55]
\addplot table[ x expr=\coordindex, y expr= \thisrowno{0} / \thisrowno{1}, col sep=space ] {./telescope2ndorder.NonCompat-2-0.plot};
\addplot table[ x expr=\coordindex, y expr= \thisrowno{0} / \thisrowno{1}, col sep=space ] {./tube3d.NonCompat-2-0.plot};
\addplot table[ x expr=\coordindex, y expr= \thisrowno{0} / \thisrowno{1}, col sep=space ] {./TA052.NonCompat-2-0.plot};
\addplot table[ x expr=\coordindex, y expr= \thisrowno{0} / \thisrowno{1}, col sep=space ] {./decorstatue.NonCompat-2-0.plot};
\addplot table[ x expr=\coordindex, y expr= \thisrowno{0} / \thisrowno{1}, col sep=space ] {./halfcircle.NonCompat-2-0.plot};
\addplot table[ x expr=\coordindex, y expr= \thisrowno{0} / \thisrowno{1}, col sep=space ] {./hole3.NonCompat-2-0.plot};
\end{axis}
\end{tikzpicture}
\begin{tikzpicture}[scale=0.405]
\begin{axis}[ xlabel={$C_{LAE}$ - SRN2}, ymin=0, ymax=0.55]
\addplot table[ x expr=\coordindex, y expr= \thisrowno{0} / (\thisrowno{1} + \thisrowno{0}), col sep=space ] {./telescope2ndorder.NonCompat-2-1.plot};
\addplot table[ x expr=\coordindex, y expr= \thisrowno{0} / (\thisrowno{1} + \thisrowno{0}), col sep=space ] {./tube3d.NonCompat-2-1.plot};
\addplot table[ x expr=\coordindex, y expr= \thisrowno{0} / (\thisrowno{1} + \thisrowno{0}), col sep=space ] {./TA052.NonCompat-2-1.plot};
\addplot table[ x expr=\coordindex, y expr= \thisrowno{0} / (\thisrowno{1} + \thisrowno{0}), col sep=space ] {./decorstatue.NonCompat-2-1.plot};
\addplot table[ x expr=\coordindex, y expr= \thisrowno{0} / (\thisrowno{1} + \thisrowno{0}), col sep=space ] {./halfcircle.NonCompat-2-1.plot};
\addplot table[ x expr=\coordindex, y expr= \thisrowno{0} / (\thisrowno{1} + \thisrowno{0}), col sep=space ] {./hole3.NonCompat-2-1.plot};
\end{axis}
\end{tikzpicture}
\caption{Realistic grids}
\end{subfigure}
\caption{Percentage of not strongly compatible faces over different threshold values.}
\label{fig:noncompat}
\end{figure}

Figure \ref{fig:noncompat} displays the amount of not strongly compatible faces in the grid. There are a few things to note. In most cases SRN2 outperforms SRN by about $10\%$ of the faces. The Kuhn-grid is able to recover the strong compatibility. Surprisingly this is reached in the inactive region (i.e. OT0) combined with SRN.

Using the abbreviations Volume(V), Longest Edge(LE), Shortest Edge(SE), Largest Face(LF), Smallest Face(SF), Face Volume(F), we examine the following geometric metrics of the grid: V/LE$^3$, V/SE$^3$, V/LF$^{3/2}$, V/SF$^{3/2}$, F/LE$^2$, F/SE$^2$ and the d-sine \cite{Eriksson1978}. We observed that (V/LE$^3$) $\approx$ (V/LF$^{2/3}$)$\times$ (F/LE$^2$) and additionally (V/LF) $\approx$ d-sine. Qualitatively the metrics behave similar, so we choose to display only the d-sine here. 

The first value in these graphs with Threshold value $C_{ILE/LAE}=-1$ denotes the unrefined grid, which is obviously the same for all Thresholds. The geometric quality indicators have been measured on the 3 times uniformly refined grid.

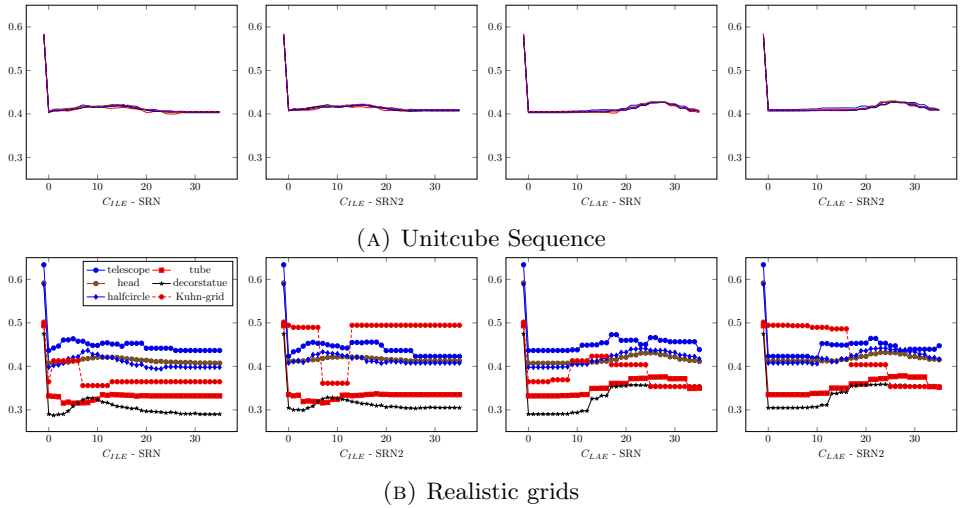
\begin{figure}[ht]
\begin{subfigure}{0.99\textwidth}
\begin{tikzpicture}[scale=0.405]
\begin{axis}[xlabel={$C_{ILE}$ - SRN},ymin=0.25,ymax=0.65]
\foreach \i in {1,...,6}{
\addplot table[ x expr=\coordindex-1, y index=21, col sep=space, mark={} ] {./tet.\i.msh.threshold-1-0.plot};
}
\end{axis}
\end{tikzpicture}
\begin{tikzpicture}[scale=0.405]
\begin{axis}[xlabel={$C_{ILE}$ - SRN2},ymin=0.25,ymax=0.65]
\foreach \i in {1,...,6}{
\addplot table[ x expr=\coordindex-1, y index=21, col sep=space, mark={}] {./tet.\i.msh.threshold-1-1.plot};
}
\end{axis}
\end{tikzpicture}
\begin{tikzpicture}[scale=0.405]
\begin{axis}[xlabel={$C_{LAE}$ - SRN},ymin=0.25,ymax=0.65]
\foreach \i in {1,...,6}{
\addplot table[ x expr=\coordindex-1, y index=21, col sep=space, mark={}] {./tet.\i.msh.threshold-2-0.plot};
}
\end{axis}
\end{tikzpicture}
\begin{tikzpicture}[scale=0.405]
\begin{axis}[xlabel={$C_{LAE}$ - SRN2},ymin=0.25,ymax=0.65]
\foreach \i in {1,...,6}{
\addplot table[ x expr=\coordindex-1, y index=21, col sep=space, mark={}] {./tet.\i.msh.threshold-2-1.plot};
}
\end{axis}
\end{tikzpicture}
\caption{Unitcube Sequence}
\end{subfigure}
\begin{subfigure}{0.99\textwidth}
\begin{tikzpicture}[scale=0.405]
\begin{axis}[legend columns=2, xlabel={$C_{ILE}$ - SRN},ymin=0.25,ymax=0.65]
\addplot table[ x expr=\coordindex-1, y index= 21, col sep=space ] {./telescope2ndorder.threshold-1-0.plot};
\addlegendentry{telescope};
\addplot table[ x expr=\coordindex-1, y index= 21, col sep=space ] {./tube3d.threshold-1-0.plot};
\addlegendentry{tube};
\addplot table[ x expr=\coordindex-1, y index= 21, col sep=space ]  {./TA052.threshold-1-0.plot};
\addlegendentry{head};
\addplot table[ x expr=\coordindex-1, y index= 21, col sep=space ]  {./decorstatue.threshold-1-0.plot};
\addlegendentry{decorstatue};
\addplot table[ x expr=\coordindex-1, y index= 21, col sep=space ]  {./halfcircle.threshold-1-0.plot};
\addlegendentry{halfcircle};
\addplot table[ x expr=\coordindex-1, y index= 21, col sep=space ]  {./hole3.threshold-1-0.plot};
\addlegendentry{Kuhn-grid};
\end{axis}
\end{tikzpicture}
\begin{tikzpicture}[scale=0.405]
\begin{axis}[xlabel={$C_{ILE}$ - SRN2},ymin=0.25,ymax=0.65]
\addplot table[ x expr=\coordindex-1, y index= 21, col sep=space ] {./telescope2ndorder.threshold-1-1.plot};
\addplot table[ x expr=\coordindex-1, y index= 21, col sep=space ] {./tube3d.threshold-1-1.plot};
\addplot table[ x expr=\coordindex-1, y index= 21, col sep=space ]  {./TA052.threshold-1-1.plot};
\addplot table[ x expr=\coordindex-1, y index= 21, col sep=space ]  {./decorstatue.threshold-1-1.plot};
\addplot table[ x expr=\coordindex-1, y index= 21, col sep=space ]  {./halfcircle.threshold-1-1.plot};
\addplot table[ x expr=\coordindex-1, y index= 21, col sep=space ]  {./hole3.threshold-1-1.plot};
\end{axis}
\end{tikzpicture}
\begin{tikzpicture}[scale=0.405]
\begin{axis}[xlabel={$C_{LAE}$ - SRN},ymin=0.25,ymax=0.65]
\addplot table[ x expr=\coordindex-1, y index= 21, col sep=space ] {./telescope2ndorder.threshold-2-0.plot};
\addplot table[ x expr=\coordindex-1, y index= 21, col sep=space ] {./tube3d.threshold-2-0.plot};
\addplot table[ x expr=\coordindex-1, y index= 21, col sep=space ]  {./TA052.threshold-2-0.plot};
\addplot table[ x expr=\coordindex-1, y index= 21, col sep=space ]  {./decorstatue.threshold-2-0.plot};
\addplot table[ x expr=\coordindex-1, y index= 21, col sep=space ]  {./halfcircle.threshold-2-0.plot};
\addplot table[ x expr=\coordindex-1, y index= 21, col sep=space ]  {./hole3.threshold-2-0.plot};
\end{axis}
\end{tikzpicture}
\begin{tikzpicture}[scale=0.405]
\begin{axis}[xlabel={$C_{LAE}$ - SRN2},ymin=0.25,ymax=0.65]
\addplot table[ x expr=\coordindex-1, y index= 21, col sep=space ] {./telescope2ndorder.threshold-2-1.plot};
\addplot table[ x expr=\coordindex-1, y index= 21, col sep=space ] {./tube3d.threshold-2-1.plot};
\addplot table[ x expr=\coordindex-1, y index= 21, col sep=space ]  {./TA052.threshold-2-1.plot};
\addplot table[ x expr=\coordindex-1, y index= 21, col sep=space ]  {./decorstatue.threshold-2-1.plot};
\addplot table[ x expr=\coordindex-1, y index= 21, col sep=space ]  {./halfcircle.threshold-2-1.plot};
\addplot table[ x expr=\coordindex-1, y index= 21, col sep=space ]  {./hole3.threshold-2-1.plot};
\end{axis}
\end{tikzpicture}
\caption{Realistic grids}
\end{subfigure}
\caption{Average d-sine}
\label{fig:dsine}
\end{figure}

Figure \ref{fig:dsine} displays the average d-sine for the investigated grids. For almost all grids we see the expected significant drop of geometric quality from the initial grid to the 3 times uniformly refined mesh and then some improvement in the active region. The only exception is the Kuhn-grid, which actually gets worse in the active region, where it is not strongly compatible anymore. Also there is no initial drop, as the mesh is based on Kuhn-simplices. For other grids LAE seems to perform a bit better than ILE, and SRN versus SRN2 does not seem to have much impact.

\begin{figure}[ht]
\begin{subfigure}{0.99\textwidth}
\begin{tikzpicture}[scale=0.4]
\begin{axis}[xlabel={$C_{ILE}$ - SRN}, ymin=35, ymax=170]
\foreach \i in {1,...,6}{
\addplot table[ x expr=\coordindex-1, y index=26, col sep=space, mark={}] {./tet.\i.msh.threshold-1-0.plot};
}
\end{axis}
\end{tikzpicture}
\begin{tikzpicture}[scale=0.4]
\begin{axis}[xlabel={$C_{ILE}$ - SRN2}, ymin=35, ymax=170]
\foreach \i in {1,...,6}{
\addplot table[ x expr=\coordindex-1, y index=26, col sep=space, mark={}] {./tet.\i.msh.threshold-1-1.plot};
}
\end{axis}
\end{tikzpicture}
\begin{tikzpicture}[scale=0.4]
\begin{axis}[xlabel={$C_{LAE}$ - SRN}, ymin=35, ymax=170]
\foreach \i in {1,...,6}{
\addplot table[ x expr=\coordindex-1, y index=26, col sep=space, mark={} ] {./tet.\i.msh.threshold-2-0.plot};
}
\end{axis}
\end{tikzpicture}
\begin{tikzpicture}[scale=0.4]
\begin{axis}[xlabel={$C_{LAE}$ - SRN2}, ymin=35, ymax=170]
\foreach \i in {1,...,6}{
\addplot table[ x expr=\coordindex-1, y index=26, col sep=space, mark={} ] {./tet.\i.msh.threshold-2-1.plot};
}
\end{axis}
\end{tikzpicture}
\caption{Unitcube Sequence}
\end{subfigure}
\begin{subfigure}{0.99\textwidth}
\begin{tikzpicture}[scale=0.4]
\begin{axis}[xlabel={$C_{ILE}$ - SRN}, ymax=250]
\addplot table[ x expr=\coordindex-1, y index= 26, col sep=space ] {./telescope2ndorder.threshold-1-0.plot};
\addplot table[ x expr=\coordindex-1, y index= 26, col sep=space ] {./tube3d.threshold-1-0.plot};
\addplot table[ x expr=\coordindex-1, y index= 26, col sep=space ]  {./TA052.threshold-1-0.plot};
\addplot table[ x expr=\coordindex-1, y index= 26, col sep=space ]  {./halfcircle.threshold-1-0.plot};
\addplot table[ x expr=\coordindex-1, y index= 26, col sep=space ]  {./hole3.threshold-1-0.plot};
\end{axis}
\end{tikzpicture}
\begin{tikzpicture}[scale=0.4]
\begin{axis}[xlabel={$C_{ILE}$ - SRN2}, ymax=250]
\addplot table[ x expr=\coordindex-1, y index= 26, col sep=space ] {./telescope2ndorder.threshold-1-1.plot};
\addplot table[ x expr=\coordindex-1, y index= 26, col sep=space ] {./tube3d.threshold-1-1.plot};
\addplot table[ x expr=\coordindex-1, y index= 26, col sep=space ]  {./TA052.threshold-1-1.plot};
\addplot table[ x expr=\coordindex-1, y index= 26, col sep=space ]  {./halfcircle.threshold-1-1.plot};
\addplot table[ x expr=\coordindex-1, y index= 26, col sep=space ]  {./hole3.threshold-1-1.plot};
\end{axis}
\end{tikzpicture}
\begin{tikzpicture}[scale=0.4]
\begin{axis}[xlabel={$C_{LAE}$ - SRN}, ymax=250]
\addplot table[ x expr=\coordindex-1, y index= 26, col sep=space ] {./telescope2ndorder.threshold-2-0.plot};
\addplot table[ x expr=\coordindex-1, y index= 26, col sep=space ] {./tube3d.threshold-2-0.plot};
\addplot table[ x expr=\coordindex-1, y index= 26, col sep=space ]  {./TA052.threshold-2-0.plot};
\addplot table[ x expr=\coordindex-1, y index= 26, col sep=space ]  {./halfcircle.threshold-2-0.plot};
\addplot table[ x expr=\coordindex-1, y index= 26, col sep=space ]  {./hole3.threshold-2-0.plot};
\end{axis}
\end{tikzpicture}
\begin{tikzpicture}[scale=0.4]
\begin{axis}[xlabel={$C_{LAE}$ - SRN2}, ymax=250]
\addplot table[ x expr=\coordindex-1, y index= 26, col sep=space ] {./telescope2ndorder.threshold-2-1.plot};
\addlegendentry{telescope};
\addplot table[ x expr=\coordindex-1, y index= 26, col sep=space ] {./tube3d.threshold-2-1.plot};
\addlegendentry{tube};
\addplot table[ x expr=\coordindex-1, y index= 26, col sep=space ]  {./TA052.threshold-2-1.plot};
\addlegendentry{head};
\addplot table[ x expr=\coordindex-1, y index= 26, col sep=space ]  {./halfcircle.threshold-2-1.plot};
\addlegendentry{halfcircle};
\addplot table[ x expr=\coordindex-1, y index= 26, col sep=space ]  {./hole3.threshold-2-1.plot};
\addlegendentry{Kuhn-grid};
\end{axis}
\end{tikzpicture}
\caption{Realistic grids without decorstatue}
\end{subfigure}
\begin{subfigure}{0.99\textwidth}
\begin{tikzpicture}[scale=0.4]
\begin{axis}[xlabel={$C_{ILE}$ - SRN}, ymax=900]
\addplot table[ x expr=\coordindex-1, y index= 26, col sep=space ]  {./decorstatue.threshold-1-0.plot};
\addlegendentry{decorstatue};
\end{axis}
\end{tikzpicture}
\begin{tikzpicture}[scale=0.4]
\begin{axis}[xlabel={$C_{ILE}$ - SRN2}, ymax=900]
\addplot table[ x expr=\coordindex-1, y index= 26, col sep=space ]  {./decorstatue.threshold-1-1.plot};
\end{axis}
\end{tikzpicture}
\begin{tikzpicture}[scale=0.4]
\begin{axis}[xlabel={$C_{LAE}$ - SRN}, ymax=900]
\addplot table[ x expr=\coordindex-1, y index= 26, col sep=space ]  {./decorstatue.threshold-2-0.plot};
\end{axis}
\end{tikzpicture}
\begin{tikzpicture}[scale=0.4]
\begin{axis}[xlabel={$C_{LAE}$ - SRN2}, ymax=900]
\addplot table[ x expr=\coordindex-1, y index= 26, col sep=space ]  {./decorstatue.threshold-2-1.plot};
\end{axis}
\end{tikzpicture}
\caption{Decorstatue}
\end{subfigure}
\caption{Maximum Number of Elements at a vertex}
\label{fig:maxvtx}
\end{figure}
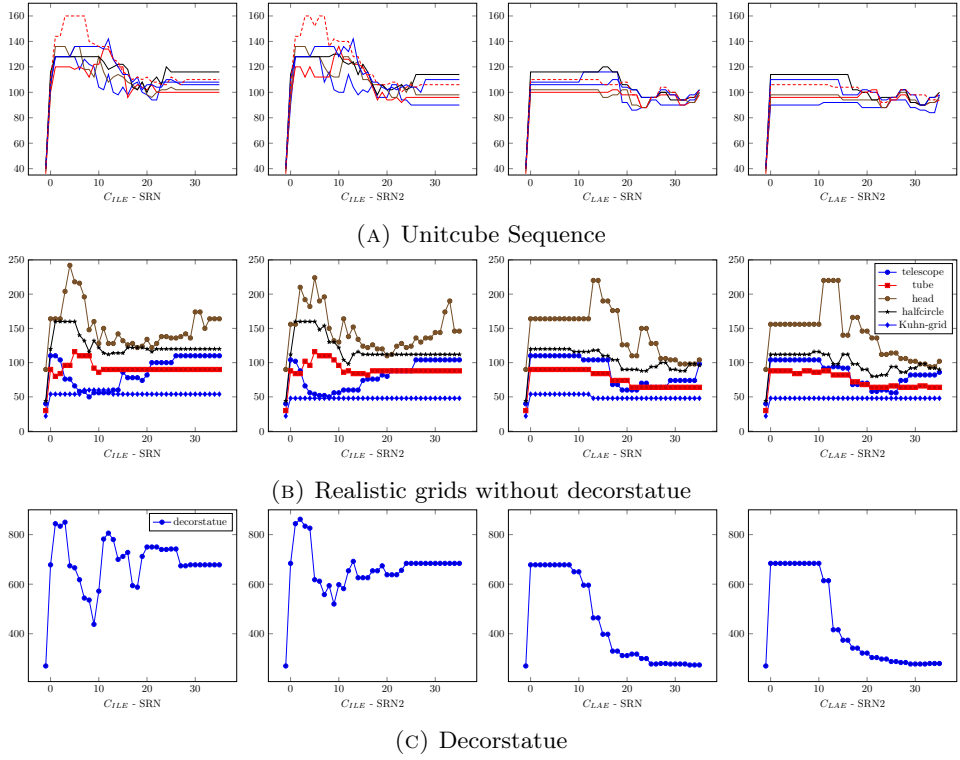

A way to investigate the minimum angle is to consider the maximum number of adjacent elements at a vertex. Figure \ref{fig:maxvtx} indicates that there is a big difference between LAE and ILE. In the active region of ILE this gets worse, while in the active region of LAE this value improves. To get an impression of the size of this value note the following calculation. In the equilateral case at every vertex there are 20 adjacent elements, uniform bisection in 3d can quadruple the number of elements at a vertex, which  yields an expected 80 elements as a maximum after 3 uniform refinements for equilateral initial grids. The best values of LAE are not far from that. 
The special grid originating from Kuhn-cubes (see Figure \ref{fig:kuhngrid}) performs best under bisection with respect to both measures. The second special case among the realistic grids is the decorstatue grid \ref{fig:decorstatue}. It has by far the most elements at a vertex in the initial triangulation and we see that ILE is not improving that, while LAE almost recovers the initial value at a high threshold.

\subsection{Conforming Closure}

We discuss the connection between the distances $d^1_\grid$, $d^2_\grid$ and $d^2_{\grid_3}$. Unfortunately the computation cost of $d^2_{\grid}$ is at least $O(n^2)$ and  $d^2_{\grid_3}$ is at least $O((8n)^2)$. Hence we did not compute these measures for all grids and possible parameters but just for a small selection. Also in the implementation we make the simplification, if refinement of an element $\elm$ is part of the conforming closure of another element $\elm'$ then the conforming closure of $\elm $ is included in the conforming closure of $\elm'$ and we do not compute the closure of $\elm$. This means the maximum conforming closure is computed exactly, while we just provide an upper bound for the average conforming closure.

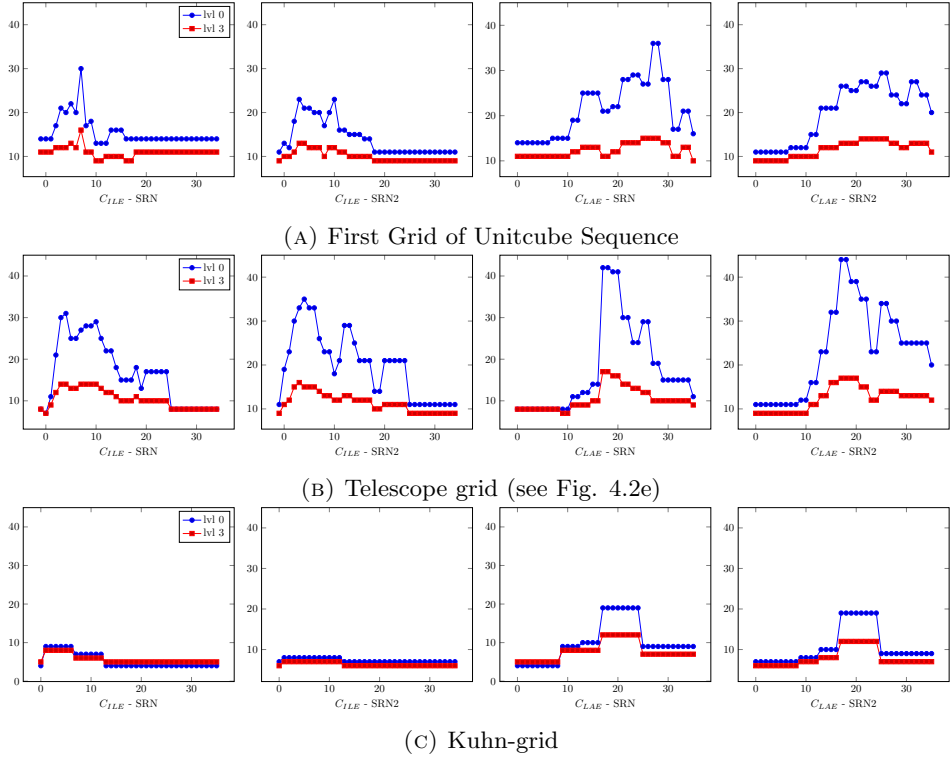
\begin{figure}[ht]
\begin{subfigure}{0.99\textwidth}
\begin{tikzpicture}[scale=0.405]
\begin{axis}[xlabel={$C_{ILE}$ - SRN}, ymax=45]
\addplot table[ x expr=\coordindex-1, y index=2, col sep=space ] {./tet.0.msh.Closure0-1-0.plot};
\addlegendentry{lvl 0};
\addplot table[ x expr=\coordindex-1, y index=2, col sep=space ] {./tet.0.msh.Closure3-1-0.plot};
\addlegendentry{lvl 3};
\end{axis}
\end{tikzpicture}
\begin{tikzpicture}[scale=0.405]
\begin{axis}[xlabel={$C_{ILE}$ - SRN2}, ymax=45]
\addplot table[ x expr=\coordindex-1, y index=2, col sep=space ] {./tet.0.msh.Closure0-1-1.plot};
\addplot table[ x expr=\coordindex-1, y index=2, col sep=space ] {./tet.0.msh.Closure3-1-1.plot};
\end{axis}
\end{tikzpicture}
\begin{tikzpicture}[scale=0.405]
\begin{axis}[xlabel={$C_{LAE}$ - SRN}, ymax=45]
\addplot table[ x expr=\coordindex, y index=2, col sep=space ] {./tet.0.msh.Closure0-2-0.plot};
\addplot table[ x expr=\coordindex, y index=2, col sep=space ] {./tet.0.msh.Closure3-2-0.plot};
\end{axis}
\end{tikzpicture}
\begin{tikzpicture}[scale=0.405]
\begin{axis}[xlabel={$C_{LAE}$ - SRN2}, ymax=45]
\addplot table[ x expr=\coordindex, y index=2, col sep=space ] {./tet.0.msh.Closure0-2-1.plot};
\addplot table[ x expr=\coordindex, y index=2, col sep=space ] {./tet.0.msh.Closure3-2-1.plot};
\end{axis}
\end{tikzpicture}
\caption{First Grid of Unitcube Sequence}
\end{subfigure}
\begin{subfigure}{0.99\textwidth}
\begin{tikzpicture}[scale=0.405]
\begin{axis}[xlabel={$C_{ILE}$ - SRN}, ymax=45]
\addplot table[ x expr=\coordindex-1, y index=2, col sep=space ] {./telescope2ndorder.Closure0-1-0.plot};
\addlegendentry{lvl 0};
\addplot table[ x expr=\coordindex-1, y index=2, col sep=space ] {./telescope2ndorder.Closure3-1-0.plot};
\addlegendentry{lvl 3};
\end{axis}
\end{tikzpicture}
\begin{tikzpicture}[scale=0.405]
\begin{axis}[xlabel={$C_{ILE}$ - SRN2}, ymax=45]
\addplot table[ x expr=\coordindex-1, y index=2, col sep=space ] {./telescope2ndorder.Closure0-1-1.plot};
\addplot table[ x expr=\coordindex-1, y index=2, col sep=space ] {./telescope2ndorder.Closure3-1-1.plot};
\end{axis}
\end{tikzpicture}
\begin{tikzpicture}[scale=0.405]
\begin{axis}[xlabel={$C_{LAE}$ - SRN}, ymax=45]
\addplot table[ x expr=\coordindex, y index=2, col sep=space ] {./telescope2ndorder.Closure0-2-0.plot};
\addplot table[ x expr=\coordindex, y index=2, col sep=space ] {./telescope2ndorder.Closure3-2-0.plot};
\end{axis}
\end{tikzpicture}
\begin{tikzpicture}[scale=0.405]
\begin{axis}[xlabel={$C_{LAE}$ - SRN2}, ymax=45]
\addplot table[ x expr=\coordindex, y index=2, col sep=space ] {./telescope2ndorder.Closure0-2-1.plot};
\addplot table[ x expr=\coordindex, y index=2, col sep=space ] {./telescope2ndorder.Closure3-2-1.plot};
\end{axis}
\end{tikzpicture}
  \caption{Telescope grid (see Fig. \ref{fig:telescope})}
\end{subfigure}
\begin{subfigure}{0.99\textwidth}
\begin{tikzpicture}[scale=0.405]
\begin{axis}[xlabel={$C_{ILE}$ - SRN}, ymax=45]
\addplot table[ x expr=\coordindex, y index= 2, col sep=space ]  {./hole3.Closure0-1-0.plot};
\addlegendentry{lvl 0};
\addplot table[ x expr=\coordindex, y index= 2, col sep=space ]  {./hole3.Closure3-1-0.plot};
\addlegendentry{lvl 3};
\end{axis}
\end{tikzpicture}
\begin{tikzpicture}[scale=0.405]
\begin{axis}[xlabel={$C_{ILE}$ - SRN2}, ymax=45]
\addplot table[ x expr=\coordindex, y index= 2, col sep=space ]  {./hole3.Closure0-1-1.plot};
\addplot table[ x expr=\coordindex, y index= 2, col sep=space ]  {./hole3.Closure3-1-1.plot};
\end{axis}
\end{tikzpicture}
\begin{tikzpicture}[scale=0.405]
\begin{axis}[xlabel={$C_{LAE}$ - SRN}, ymax=45]
\addplot table[ x expr=\coordindex, y index= 2, col sep=space ]  {./hole3.Closure0-2-0.plot};
\addplot table[ x expr=\coordindex, y index= 2, col sep=space ]  {./hole3.Closure3-2-0.plot};
\end{axis}
\end{tikzpicture}
\begin{tikzpicture}[scale=0.405]
\begin{axis}[xlabel={$C_{LAE}$ - SRN2}, ymax=45]
\addplot table[ x expr=\coordindex, y index= 2, col sep=space ]  {./hole3.Closure0-2-1.plot};
\addplot table[ x expr=\coordindex, y index= 2, col sep=space ]  {./hole3.Closure3-2-1.plot};
\end{axis}
\end{tikzpicture}
\caption{Kuhn-grid}
\end{subfigure}
\caption{Average conforming closure}
\label{fig:avgclos}
\end{figure}

\begin{figure}[ht]
\begin{subfigure}{0.99\textwidth}
\begin{tikzpicture}[scale=0.4]
\begin{axis}[xlabel={$C_{ILE}$ - SRN}, ymax=180]
\addplot table[ x expr=\coordindex, y index=1, col sep=space ] {./tet.0.msh.Closure0-1-0.plot};
\addlegendentry{lvl 0};
\addplot table[ x expr=\coordindex, y index=1, col sep=space ] {./tet.0.msh.Closure3-1-0.plot};
\addlegendentry{lvl 3};
\end{axis}
\end{tikzpicture}
\begin{tikzpicture}[scale=0.4]
\begin{axis}[xlabel={$C_{ILE}$ - SRN2}, ymax=180]
\addplot table[ x expr=\coordindex, y index=1, col sep=space ] {./tet.0.msh.Closure0-1-1.plot};
\addplot table[ x expr=\coordindex, y index=1, col sep=space ] {./tet.0.msh.Closure3-1-1.plot};
\end{axis}
\end{tikzpicture}
\begin{tikzpicture}[scale=0.4]
\begin{axis}[xlabel={$C_{LAE}$ - SRN}, ymax=180]
\addplot table[ x expr=\coordindex, y index=1, col sep=space ] {./tet.0.msh.Closure0-2-0.plot};
\addplot table[ x expr=\coordindex, y index=1, col sep=space ] {./tet.0.msh.Closure3-2-0.plot};
\end{axis}
\end{tikzpicture}
\begin{tikzpicture}[scale=0.4]
\begin{axis}[xlabel={$C_{LAE}$ - SRN2}, ymax=180]
\addplot table[ x expr=\coordindex, y index=1, col sep=space ] {./tet.0.msh.Closure0-2-1.plot};
\addplot table[ x expr=\coordindex, y index=1, col sep=space ] {./tet.0.msh.Closure3-2-1.plot};
\end{axis}
\end{tikzpicture}
\caption{First grid of Unitcube Sequence}
\end{subfigure}
\begin{subfigure}{0.99\textwidth}
\begin{tikzpicture}[scale=0.4]
\begin{axis}[xlabel={$C_{ILE}$ - SRN}, ymax=180]
\addplot table[ x expr=\coordindex, y index=1, col sep=space ] {./telescope2ndorder.Closure0-1-0.plot};
\addlegendentry{lvl 0};
\addplot table[ x expr=\coordindex, y index=1, col sep=space ] {./telescope2ndorder.Closure3-1-0.plot};
\addlegendentry{lvl 3};
\end{axis}
\end{tikzpicture}
\begin{tikzpicture}[scale=0.4]
\begin{axis}[xlabel={$C_{ILE}$ - SRN2}, ymax=180]
\addplot table[ x expr=\coordindex, y index=1, col sep=space ] {./telescope2ndorder.Closure0-1-1.plot};
\addplot table[ x expr=\coordindex, y index=1, col sep=space ] {./telescope2ndorder.Closure3-1-1.plot};
\end{axis}
\end{tikzpicture}
\begin{tikzpicture}[scale=0.4]
\begin{axis}[xlabel={$C_{LAE}$ - SRN}, ymax=180]
\addplot table[ x expr=\coordindex, y index=1, col sep=space ] {./telescope2ndorder.Closure0-2-0.plot};
\addplot table[ x expr=\coordindex, y index=1, col sep=space ] {./telescope2ndorder.Closure3-2-0.plot};
\end{axis}
\end{tikzpicture}
\begin{tikzpicture}[scale=0.4]
\begin{axis}[xlabel={$C_{LAE}$ - SRN2}, ymax=180]
\addplot table[ x expr=\coordindex, y index=1, col sep=space ] {./telescope2ndorder.Closure0-2-1.plot};
\addplot table[ x expr=\coordindex, y index=1, col sep=space ] {./telescope2ndorder.Closure3-2-1.plot};
\end{axis}
\end{tikzpicture}
\caption{Telescope Grid}
\end{subfigure}
\begin{subfigure}{0.99\textwidth}
\begin{tikzpicture}[scale=0.4]
\begin{axis}[xlabel={$C_{ILE}$ - SRN}, ymax=180]
\addplot table[ x expr=\coordindex, y index= 1, col sep=space ]  {./hole3.Closure0-1-0.plot};
\addlegendentry{lvl 0};
\addplot table[ x expr=\coordindex, y index= 1, col sep=space ]  {./hole3.Closure3-1-0.plot};
\addlegendentry{lvl 3};
\end{axis}
\end{tikzpicture}
\begin{tikzpicture}[scale=0.4]
\begin{axis}[xlabel={$C_{ILE}$ - SRN2}, ymax=180]
\addplot table[ x expr=\coordindex, y index= 1, col sep=space ]  {./hole3.Closure0-1-1.plot};
\addplot table[ x expr=\coordindex, y index= 1, col sep=space ]  {./hole3.Closure3-1-1.plot};
\end{axis}
\end{tikzpicture}
\begin{tikzpicture}[scale=0.4]
\begin{axis}[xlabel={$C_{LAE}$ - SRN}, ymax=180]
\addplot table[ x expr=\coordindex, y index= 1, col sep=space ]  {./hole3.Closure0-2-0.plot};
\addplot table[ x expr=\coordindex, y index= 1, col sep=space ]  {./hole3.Closure3-2-0.plot};
\end{axis}
\end{tikzpicture}
\begin{tikzpicture}[scale=0.4]
\begin{axis}[xlabel={$C_{LAE}$ - SRN2}, ymax=180]
\addplot table[ x expr=\coordindex, y index= 1, col sep=space ]  {./hole3.Closure0-2-1.plot};
\addplot table[ x expr=\coordindex, y index= 1, col sep=space ]  {./hole3.Closure3-2-1.plot};
\end{axis}
\end{tikzpicture}
\caption{Kuhn-grid}
\end{subfigure}
\caption{$d^2_\grid$ and $d^2_{\grid_3}$ }
\label{fig:maxclos}
\end{figure}

Figure \ref{fig:avgclos} displays an upper bound for the average conforming closure of each element in a small grid of the tetgen sequence and the special grid, which becomes strongly compatible. Figure \ref{fig:maxclos} actually displays $d^1_\grid$ and $d^1_{\grid_3}$, i.e. the maximum conforming closure for both grids. Both times we measure both on the initial grid (blue) and on the 3 times uniformly refined grid (red). The behavior is almost always better in the 3 times uniformly refined case. The conforming closure is more costly in the active region, as we additionally introduce refinement propagation by having neighbors of different type. In the case of SRN we see for the tetgen mesh that the maximum conforming closure actually diminishes in the active region, but SRN2 performs better in total. The Kuhn-grid of course performs best in the inactive region, where it actually becomes strongly compatible.

\subsection{Algorithm complexity}

We also studied the runtime of the algorithm over the sequence of unit cube meshes that we 
used to study the threshold values. Figure \ref{fig:runtime} shows that the runtime 
actually is $O(n \log  n)$ with $n$ being the number of elements. 
The measured runtime includes recalculating the grid neighborhood information, 
checking whether a grid is compatible initially, setting up $\vertices_0,\vertices_1$, sorting the mesh, 
and finally checking how compatible the mesh is. 
\new{The algorithm can be implemented in $O(n)$, if the recalculation of the neighborhood information is not necessary. 
In the current implementation the association of neighboring elements is done via comparison of vertex indices of faces and 
storage of those in a \texttt{std::map} with $O(\log n)$ member access, where $n$ 
is the number of faces. Therefore, the overall algorithm complexity becomes $O(n \log n)$.
In our tests we have applied the various algorithms to meshes with about one
million grid cells in under $10$ seconds which we consider sufficiently fast, especially since the algorithm has to be applied only once.}

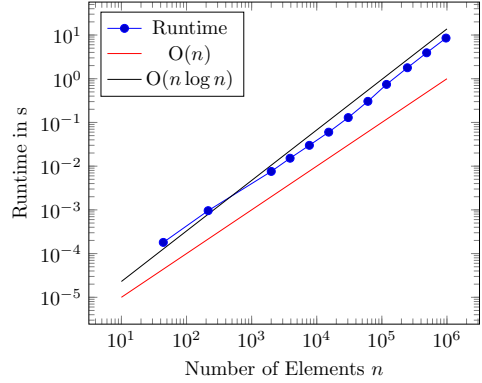
\begin{figure}[!ht]
  \begin{center}
\begin{tikzpicture}[scale=0.75]
\begin{loglogaxis}[legend style={at={(0.03,0.97)},anchor=north west}, xlabel = Number of Elements $n$, ylabel = Runtime in s]
\addplot table[ x index=0, y index= 1, col sep=space ]  {./bisect.gnu};
\addlegendentry{Runtime};
\addplot[mark={},red] coordinates { (10^1,10^-5) (10^6,10^0) } ;
\addlegendentry{O($n$)};
\addplot[mark={},black] coordinates { (10^1,2.3*10^-5) (10^6,13.8*10^0) } ;
\addlegendentry{O($n \log n$)};
\end{loglogaxis}
\end{tikzpicture}
\caption{Runtime of the algorithm for different grid sizes.}
\label{fig:runtime}
  \end{center}
\end{figure}

\section{Summary and Outlook} \label{sec:summary}

We have presented a weak compatibility condition that is sufficient for the iterative \NVB algorithm to terminate. Additionally we provide a simple algorithm to relabel any $d$-dimensional grid to fulfil said condition with effort $O(n)$. The methods we designed try to achieve a compromise between LEB and \NVB. We suggest to use (LAE)/(SRN2) with a parameter around 27 for grids created by mesh generators that aim at equilateral tetrahedrons. For meshes created by Kuhn-cubes it is best to just use the strongly compatible enumeration and not to try to improve the behaviour.

It is still an open question whether the \emph{weak compatibility condition} is enough to obtain 
an estimate like the one in~\cite{Stevenson:08} for the computational effort of the conforming closure
 or if this can be achieved with some extra mild conditions in the mesh.

\section*{Acknowledgements}

The author Martin Alk\"amper would like to thank the German Research Foundation (DFG) for financial support of the project within the Cluster of Excellence in Simulation Technology (EXC 310/1) at the University of Stuttgart.

Fernando Gaspoz acknowledges the German Research Foundation (DFG) for financial support of the project DFG 814/7-1 at the University of Stuttgart.

Robert Kl\"ofkorn acknowledges the Research Council of Norway and 
the industry partners;  ConocoPhillips Skandinavia AS, 
Aker BP ASA, Eni Norge AS, Maersk Oil Norway AS, 
DONG Energy A/S, Denmark, Statoil Petroleum AS, ENGIE E\&P NORGE AS, 
Lundin Norway AS, Halliburton AS, Schlumberger Norge AS, 
Wintershall Norge AS, DEA Norge AS of 
The National IOR Centre of Norway for support.

\end{document}